\documentclass[11pt]{amsart}
\usepackage{amsfonts}

\newtheorem{theorem}{Theorem}
\theoremstyle{plain}

\newtheorem{corollary}{Corollary}

\newtheorem{definition}{Definition}
\newtheorem{example}{Example}

\newtheorem{lemma}{Lemma}

\newtheorem{proposition}{Proposition}
\newtheorem{remark}{Remark}

\numberwithin{equation}{section}
\input{tcilatex}

\begin{document}
\title[Convergence rate of weak Euler schemes]{On the rate of convergence of
simple and jump-adapted weak Euler schemes for L\'{e}vy driven SDEs}
\author{R. Mikulevicius}
\address{University of Southern California, Los Angeles, CA}
\date{February 20, 2012}
\subjclass{60J75 (Primary) 60J60, 60H30, 45K05, 35S10}
\keywords{Parabolic integro-differntial equations, Weak Euler scheme,
Approximate and jump-adapted Euler schemes}

\begin{abstract}
The paper studies the rate of convergence of a weak Euler approximation for
solutions to possibly completely degenerate SDEs driven by L\'{e}vy
processes, with H\"{o}lder-continuous coefficients. It investigates the
dependence of the rate on the regularity of coefficients and driving
processes and its robustness to the approximation of the increments of the
driving process. A convergence rate is derived for some approximate
jump-adpted Euler scheme as well.
\end{abstract}

\maketitle

\section{\textrm{Introduction}}

The paper studies the weak Euler approximation for solutions to possibly
completely degenerate SDEs driven by L\'{e}vy processes. As in \cite{mikz3},
the main goal is to investigate the dependence of the convergence rate on
the regularity of coefficients and driving processes. In addition, we
consider the robustness of the results to the approximation of the law of
the increments of the driving noise in the whole scale of time
discretization errors.

Let $(\Omega ,\mathcal{F},\mathbf{P})$ be a complete probability space with
a filtration $\mathbb{F}=\{\mathcal{F}_{t}\}_{t\in \lbrack 0,T]}$ of $\sigma 
$-algebras satisfying the usual conditions and $\alpha \in (0,2]$ be fixed.
Consider the following model in $\mathbf{R}^{d}$: 
\begin{equation}
X_{t}=X_{0}+\int_{0}^{t}a(X_{s})ds+\int_{0}^{t}b(X_{s})dW_{s}+%
\int_{0}^{t}G(X_{s-})dZ_{s},t\in \lbrack 0,T],  \label{one}
\end{equation}%
where $a(x)=(a^{i}(x))_{1\leq i\leq d}$, $b(x)=(b^{ij}(x))_{1\leq i\leq
d,1\leq j\leq n}$, $G(x)=(G^{ij}(x))_{1\leq i\leq d,1\leq j\leq m}$, $x\in 
\mathbf{R}^{d}$ are measurable and bounded, with $a=0$ if $\alpha \in (0,1)$
and $b=0$ if $\alpha \in (0,2).$ The process $W_{s}$ is a standard Wiener in 
$\mathbf{R}^{n}$. The last term is driven by $Z=\{Z_{t}\}_{t\in \lbrack
0,T]} $, an $m$-dimensional L\'{e}vy process whose characteristic function
is $\exp \left\{ t\eta (\xi )\right\} $ with 
\begin{equation*}
\eta (\xi )=\int_{\mathbf{R}_{0}^{m}}\big[e^{i(\xi ,\upsilon )}-1-i\chi
_{\alpha }(\upsilon )(\xi ,\upsilon )\big]\pi (d\upsilon ),
\end{equation*}%
where $\chi _{\alpha }(\upsilon )=\chi _{\left\{ |\upsilon |\leq 1\right\} }%
\mathbf{1}_{\{\alpha \in (1,2]\}}$. Hence, 
\begin{equation}
Z_{t}=\int_{0}^{t}\int (1-\chi _{\alpha }(\upsilon ))\upsilon p(ds,d\upsilon
)+\int_{0}^{t}\int \chi _{\alpha }(\upsilon )\upsilon q(ds,d\upsilon ), 
\notag
\end{equation}%
where $p(dt,d\upsilon )$ is a Poisson point measure on $[0,\infty )\times 
\mathbf{R}_{0}^{m}$ ($\mathbf{R}_{0}^{m}=\mathbf{R}^{m}\backslash \{0\}$)
with $\mathbf{E}[p(dt,d\upsilon )]=\pi (d\upsilon )dt$, and $q(dt,d\upsilon
)=p(dt,d\upsilon )-\pi (d\upsilon )dt$ is the centered Poisson measure. It
is assumed that $Z_{t}$ is a L\'{e}vy process of order $\alpha :$ 
\begin{equation*}
\int (|\upsilon |^{\alpha }\wedge 1)\pi (d\upsilon )<\infty .
\end{equation*}

Let the time discretization $\{\tau _{i},i=0,\ldots ,n_{T}\}$ of the
interval $[0,T]$ with maximum step size $\delta >0$ be a partition of $[0,T]$
such that $0=\tau _{0}<\tau _{1}<\dots <\tau _{n_{T}}=T$ and $\max_{i}(\tau
_{i}-\tau _{i-1})\leq \delta .$ The Euler approximation of $X$ is an $%
\mathbb{F}$-adapted stochastic process $Y=\{Y_{t}\}_{t\in \lbrack 0,T]}$
defined by the stochastic equation%
\begin{equation}
Y_{t}=X_{0}+\int_{0}^{t}a(Y_{\tau _{i_{s}}})ds+\int_{0}^{t}b(Y_{\tau
_{i_{s}}})dW_{s}+\int_{0}^{t}G(Y_{\tau _{i_{s}}})dZ_{s},t\in \lbrack 0,T],
\label{du}
\end{equation}%
where $\tau _{i_{s}}=\tau _{i}$\ if $s\in \lbrack \tau _{i},\tau
_{i+1}),i=0,\ldots ,n_{T}-1.$ Contrary to those in (\ref{one}), the
coefficients in (\ref{du}) are piecewise constants in each time interval of $%
[\tau _{i},\tau _{i+1}).$

The weak Euler approximation $Y$\ is said to converge with order $\kappa >0$%
\ if for each bounded smooth function $g$ with bounded derivatives, there
exists a constant $C$, depending only on $g$, such that 
\begin{equation*}
|\mathbf{E}g(Y_{T})-\mathbf{E}g(X_{T})|\leq C\delta ^{\kappa },
\end{equation*}%
where $\delta >0$\ is the maximum step size of the time discretization.

The weak Euler approximation of stochastic differential equations with
smooth coefficients and $G=0$ has been consistently studied. For diffusion
processes, Milstein was one of the first to investigate the order of weak
convergence and derived $\kappa =1$~\cite{Mil79, Mil86}. Talay considered a
class of the second order approximations for diffusion processes~\cite%
{Tal84, Tal86}. For It\^{o} processes with jump components (a finite number
of jumps in a finite interval), it was shown in \cite{MiP88} the first-order
convergence in the case in which the coefficient functions possess
fourth-order continuous derivatives~. Platen and Kloeden \& Platen studied
not only Euler but also higher order approximations~\cite{KlP00, Pla99} and
references therein.

Protter \& Talay (\cite{PrT97}) analyzed the weak Euler approximation for (%
\ref{one}) with $\alpha =2$. They proved that the order of convergence is $%
\kappa =1,$ provided that $G,b,a$ and $g$\ have four bounded derivatives and
the L\'{e}vy measure of $Z$\ has finite moments of the order $\mu =8$. In
this paper we show that $\kappa =1$ can be achieved when $\mu =4$ and there
still is some order of convergence for $\mu \in (2,4]$. Moreover, we assume $%
\beta $-Lipshitz continuity of the coefficients and $g$ and derive that for $%
\alpha <\beta \leq \mu \leq 2\alpha $ the order of convergence $\kappa =%
\frac{\beta }{\alpha }-1$. In particular, when $\beta =\mu =2\alpha $ with $%
\alpha \in (0,2)$ (the diffusion part is absent), the convergence order is
still $\kappa =1$.\ 

As in \cite{MiP911} and \cite{mikz3}, this paper employs the idea of Talay
(see \cite{Tal84}) and uses the solution to the backward Kolmogorov equation
associated with $X_{t},$ It\^{o}'s formula, and one-step estimates. Since
one step estimates were derived in \cite{mikz3}, the main difficulty is to
solve the degenerate backward Kolmogorov equation in Lipschitz classes (see
Theorem \ref{thm:StoCP} below). We obtain the solution of the degenerate
equation as a limit of solutions to regularized (nondegenerate) equations.
Although the solution to (\ref{one}) is strong and probabilistic arguments
are applied for the uniform Lipshitz estimates of the approximating
sequence, contrary to \cite{PrT97}, we do not use derivatives of the
stochastic flows.

If (\ref{one}) has a nondegenerate main part, some assumptions imposed can
be relaxed (see \cite{mikz3}, \cite{MiP911}, Kubilius \& Platen and Platen
\& Bruti-Liberati \cite{KuP01, PlB10}). More complex and higher order
schemes were studied and discussed, for example, by Cont and Tankov,
Jourdain and Kohatsu-Higa (see \cite{cont}, \cite{jh} and references
therein).

Motivated by the difficulty to approximate the increments of the driving
processes, Jacod, Kurtz, M\'{e}l\'{e}ard and Protter in \cite{jkmp}, studied
the approximated Euler scheme where the increments of $Z$ are substituted by
i.i.d. random variables that are easier to simulate. There are two sources
of errors in this case. One comes from time discretization and the other one
from substitution. We extend some of the results in \cite{jkmp} to the whole
rate scale and show that the errors add up. In particular, the driving
process $Z$ can be replaced with a Levy process $\tilde{Z}$ having finite
number of jumps in $[0,T]$ by possibly cutting small jumps of $Z$ and
sometimes replacing them with a Wiener process or drift. In addition, we
consider a simple jump-adapted Euler scheme and show that presence of $%
\tilde{Z}$-jump moments in the partition $\left\{ \tau _{i}\right\} $
influences the convergence rate. The approximation itself is simpler and
assumptions imposed are different than those introduced by Kohatsu-Higa and
Tankov in \cite{kot} (see references therein as well) for a more
sophisticated (higher order) jump-adapted scheme.

The paper is organized as follows. In Section 2, some notation is
introduced, the main results stated and the proof of the main theorem is
outlined. In Section 3, we present the essential technical results about
backward degenerate Kolmogorov equation, followed by the proof of the main
theorem in Section 4. The robustness of the approximation and jump-adapted
Euler scheme is considered as well. In the last section we discuss the
optimality of the imposed assumptions.

\section{Notation and Main Result}



Denote $H=[0,T]\times \mathbf{R}^{d}$, $\mathbf{N}=\{0,1,2,\ldots \}$, $%
\mathbf{R}_{0}^{d}=\mathbf{R}^{d}\backslash \{0\}$. For $x,y\in \mathbf{R}%
^{d}$, write $(x,y)=\sum_{i=1}^{d}x_{i}y_{i}$. For $(t,x)\in H,$ multiindex $%
\gamma \in \mathbf{N}^{d}$ with $D^{\gamma }=\frac{\partial ^{|\gamma |}}{%
\partial x_{1}^{\gamma _{1}}\ldots \partial x_{d}^{\gamma _{d}}}$, and $%
i,j=1,\ldots ,d$, denote 
\begin{eqnarray*}
\partial _{t}u(t,x) &=&\frac{\partial }{\partial t}u(t,x),\ D^{k}u(t,x)=\big(%
D^{\gamma }u(t,x)\big)_{|\gamma |=k},k\in \mathbf{N}\text{,} \\
\partial _{i}u(t,x) &=&u_{x_{i}}(t,x)=\frac{\partial }{\partial x_{i}}%
u(t,x),\ \partial _{ij}^{2}u(t,x)=u_{x_{i}x_{j}}(t,x)=\frac{\partial ^{2}}{%
\partial x_{i}x_{j}}u(t,x), \\
\partial _{x}u(t,x) &=&\nabla u(t,x)=\nabla _{x}u(t,x)=\big(\partial
_{1}u(t,x),\dots ,\partial _{d}u(t,x)\big), \\
\Delta u(t,x) &=&\sum_{i=1}^{d}u_{x_{i}x_{i}}(t,x).
\end{eqnarray*}%
For a smooth function $v$ on $\mathbf{R}^{d}$ and $k\in \mathbf{N}$, denote 
\begin{equation*}
v^{(k)}(x;\xi ^{1},\ldots ,\xi ^{k})=\sum_{i_{1},\ldots
,i_{k}=1}^{d}v_{x_{i_{1}}\ldots x_{i_{k}}}(x)\xi _{i_{1}}^{1}\ldots \xi
_{i_{k}}^{k},x,\xi ^{i}\in \mathbf{R}^{d},i=1,\ldots ,k.
\end{equation*}%
In particular, $v^{(1)}(x;\xi )=(\nabla v(x),\xi ),x,\xi \in \mathbf{R}^{d}$.

For $\beta =[\beta ]^{-}+\left\{ \beta \right\} ^{+}>0$, where $[\beta
]^{-}\in \mathbf{N}$ and $\left\{ \beta \right\} ^{+}\in (0,1]$, let $\tilde{%
C}^{\beta }(H)$ denote the Lipschitz space of measurable functions $u$ on $H$
such that the norm 
\begin{equation*}
|u|_{\beta }=\sum_{|\gamma |\leq \lbrack \beta ]^{-}}|D_{x}^{\gamma
}u(t,x)|_{0}+\sup_{\substack{ |\gamma |=[\beta ]^{-},  \\ t,x\neq \tilde{x}}}%
\frac{|D_{x}^{\gamma }u(t,x)-D_{x}^{\gamma }u(t,\tilde{x})|}{|x-\tilde{x}%
|^{\{\beta \}^{+}}}
\end{equation*}%
is finite, where $|v|_{0}=\sup_{(t,x)\in H}|v(t,x)|.$ We denote $\tilde{C}%
^{\beta }(\mathbf{R}^{d})$ the corresponding function space on $\mathbf{R}%
^{d}.$

$C=C(\cdot ,\ldots ,\cdot )$ denotes constants depending only on quantities
appearing in parentheses. In a given context the same letter is (generally)
used to denote different constants depending on the same set of arguments.

The main result of this paper is the following statement.

\begin{theorem}
\label{thm:main}Let $\alpha <\beta \leq \mu \leq 2\alpha $, and assume $%
a^{i},b^{ij}\in \tilde{C}^{\beta }(\mathbf{R}^{d})$, $G^{ij}\in \tilde{C}%
^{\beta \vee 1}(\mathbf{R}^{d}),$ and 
\begin{equation*}
\int_{|\upsilon |\leq 1}|\upsilon |^{\alpha }d\pi +\int_{\left\vert \upsilon
\right\vert \geq 1}|\upsilon |^{\mu }\pi (d\upsilon )<\infty ,
\end{equation*}%
where $\pi $ is the L\'{e}vy measure of the driving process $Z$. Then there
is a constant $C$ such that for all $g\in \tilde{C}^{\beta }(\mathbf{R}^{d})$%
\begin{equation*}
|\mathbf{E}g(Y_{T})-\mathbf{E}g(X_{T})|\leq C|g|_{\beta }\delta ^{\frac{%
\beta }{\alpha }-1}.
\end{equation*}
\end{theorem}

Applying Theorem~\ref{thm:main} to the case $\alpha =2$ we have an \ obvious
consequence in the jump-diffusion case.

\begin{corollary}
\label{co1}Consider the jump-diffusion case $(\alpha =2)$ 
\begin{equation*}
X_{t}=X_{0}+\int_{0}^{t}a(X_{s})ds+\int_{0}^{t}b(X_{s})dW_{s}+%
\int_{0}^{t}G(X_{s-})dZ_{s},t\in \lbrack 0,T].
\end{equation*}%
Let $2<\beta \leq \mu \leq 4$. Assume $a,b^{ij},G^{ij}\in \tilde{C}^{\beta }(%
\mathbf{R}^{d})$ and

\begin{equation*}
\int_{|\upsilon |\leq 1}|\upsilon |^{2}\pi (d\upsilon )+\int_{|\upsilon
|>1}|\upsilon |^{\mu }\pi (d\upsilon )<\infty .
\end{equation*}%
Then there is a constant $C$ such that for all $g\in \tilde{C}^{\beta }(%
\mathbf{R}^{d})$%
\begin{equation*}
|\mathbf{E}g(Y_{T})-\mathbf{E}g(X_{T})|\leq C|g|_{\beta }\delta ^{\frac{%
\beta }{2}-1}.
\end{equation*}
\end{corollary}

An immediate extension of Theorem \ref{thm:main} (for the test function $%
g\in \tilde{C}^{\nu }(\mathbf{R}^{d})$ with $\nu \in (0,\beta ]$) is the
following statement.

\begin{corollary}
\label{co2}Let $\alpha <\beta \leq \mu \leq 2\alpha $, and assume $%
a^{i},b^{ij}\in \tilde{C}^{\beta }(\mathbf{R}^{d})$, $G^{ij}\in \tilde{C}%
^{\beta \vee 1}(\mathbf{R}^{d}),$ and 
\begin{equation*}
\int_{|\upsilon |\leq 1}|\upsilon |^{\alpha }d\pi +\int_{\left\vert \upsilon
\right\vert \geq 1}|\upsilon |^{\mu }\pi (d\upsilon )<\infty ,
\end{equation*}%
where $\pi $ is the L\'{e}vy measure of the driving process $Z$. Let $\nu
\in (0,\beta ]$. Then there is a constant $C$ such that for all $g\in \tilde{%
C}^{\nu }(\mathbf{R}^{d})$%
\begin{equation*}
|\mathbf{E}g(Y_{T})-\mathbf{E}g(X_{T})|\leq C|g|_{\nu }\delta ^{\nu (\frac{1%
}{\alpha }-\frac{1}{\beta })}.
\end{equation*}
\end{corollary}

\begin{remark}
In particular, if $\alpha \in \lbrack 1,2]$, $\mu =\beta =2\alpha $ and $g$
is Lipshitz ($\nu =1$), then the convergence rate $\kappa =\frac{1}{2\alpha }%
.$
\end{remark}

\subsection{Approximate simple Euler scheme}

Following \cite{jkmp}, for $\sigma \in (0,1),\delta >0,$ we choose a time
discretization $\left\{ \tau _{i}\right\} $ and replace the increments of
the driving \ process $Z_{\tau _{i+1}}-Z_{\tau _{i}}$ in \ref{du} by $%
\mathcal{F}_{\tau _{i}}$-conditionally independent random variables $\zeta
_{i},i=0,\ldots ,n_{T}-1$. We assume that there is a function $\phi (\sigma
) $ such that $\lim_{\sigma \rightarrow 0}\phi (\sigma )=0$ and for $%
i=0,\ldots ,n_{T}-1,$%
\begin{equation}
\left\vert \mathbf{E[}h(Z_{\tau _{i+1}}-Z_{\tau _{i}})-h(\zeta _{i+1})|%
\mathcal{F}_{\tau _{i}}]\right\vert \leq C|h|_{\beta }\phi (\sigma )(\tau
_{i+1}-\tau _{i}),h\in \tilde{C}^{\beta }(\mathbf{R}^{d}).  \label{kas}
\end{equation}%
with some constant $C,$ independent of $\sigma ,\delta $ and $h$. Let $\xi
_{t}=0$ if $0\leq t<\tau _{1},\xi _{t}=\zeta _{i}$ if $t_{i}\leq
t<t_{i+1},i=1,\ldots ,n_{T}-1$. We still assume that $\max_{i}(\tau
_{i+1}-\tau _{i})\leq \delta $ and approximate $X_{t}$ by%
\begin{equation}
\tilde{Y}_{t}=X_{0}+\int_{0}^{t}a(\tilde{Y}_{\tau _{i_{s}}})ds+\int_{0}^{t}b(%
\tilde{Y}_{\tau _{i_{s}}})dW_{s}+\int_{0}^{t}G(\tilde{Y}_{\tau
_{i_{s}}})d\xi _{s},t\in \lbrack 0,T].  \label{trys}
\end{equation}%
In this case $\tilde{Y}_{t}$ depends on $\delta $ and $\sigma .$

In the following example we approximate the increments of $Z_{t}$ by the
increments of a L\'{e}vy process with finite number of jumps in $[0,T].$
This approximation is constructed by cutting small jumps of $Z_{t}$. We
replace the small jumps part by appropriately chosen drift if $\alpha <\beta
\in (1,2],\alpha \in (0,1].$ If $\alpha <\beta \in (2,3],\alpha \in (1,2],$
the small jumps part is replaced by a Wiener process. Given $\sigma \in
(0,1) $, we denote $B^{\sigma }$ the square root of the positive definite $%
m\times m\,$-matrix $\left( \int_{|\upsilon |\leq \sigma }\upsilon
_{i}\upsilon _{j}d\pi \right) _{1\leq i,j\leq m}$. Let $\tilde{W}_{t}$ be a
standard independent Wiener process in $\mathbf{R}^{m}$.

\begin{example}
\label{ex1}For $\sigma \in (0,1)$ we approximate 
\begin{equation*}
Z_{t}=\int_{0}^{t}\int (1-\chi _{\alpha }(\upsilon ))\upsilon p(ds,d\upsilon
)+\int_{0}^{t}\int \chi _{\alpha }(\upsilon )\upsilon q(ds,d\upsilon ),t\in
\lbrack 0,T],
\end{equation*}%
by%
\begin{equation*}
\tilde{Z}_{t}=Z_{t}^{\sigma }+R_{t}^{\sigma },
\end{equation*}%
with%
\begin{equation*}
Z_{t}^{\sigma }=\int_{0}^{t}\int_{|\upsilon |>\sigma }(1-\chi _{\alpha
}(\upsilon ))\upsilon p(ds,d\upsilon )+\int_{0}^{t}\int_{|\upsilon |>\sigma
}\chi _{\alpha }(\upsilon )\upsilon q(ds,d\upsilon )
\end{equation*}%
and 
\begin{equation*}
R_{t}^{\sigma }=\left\{ 
\begin{array}{ll}
t\int_{|\upsilon |\leq \sigma }\upsilon \pi (d\upsilon ) & \text{if }\alpha
<\beta \in (1,2],\alpha \in (0,1], \\ 
&  \\ 
B^{\sigma }\tilde{W}_{t} & \text{if }\alpha <\beta \in (2,4],\alpha \in
(1,2], \\ 
&  \\ 
0 & \text{otherwise.}%
\end{array}%
\right.
\end{equation*}%
In this case (see Lemma \ref{lemn2} below) (\ref{kas}) holds with 
\begin{equation*}
\phi (\sigma )=\int_{|\upsilon |\leq \sigma }|\upsilon |^{\beta \wedge 3}d\pi
\end{equation*}%
and 
\begin{equation}
\zeta _{i+1}=\tilde{Z}_{\tau _{i+1}}-\tilde{Z}_{\tau _{i}},i=0,\ldots
,n_{T}-1.  \label{apf0}
\end{equation}
\end{example}

We show that time discretization and substitution errors add up.

\begin{theorem}
\label{main2}Let $\alpha <\beta \leq \mu \leq 2\alpha $, and let $%
a^{i},b^{ij}\in \tilde{C}^{\beta }(\mathbf{R}^{d})$, $G^{ij}\in \tilde{C}%
^{\beta \vee 1}(\mathbf{R}^{d}),$ and 
\begin{equation*}
\int_{|\upsilon |\leq 1}|\upsilon |^{\alpha }d\pi +\int_{\left\vert \upsilon
\right\vert \geq 1}|\upsilon |^{\mu }\pi (d\upsilon )<\infty ,
\end{equation*}%
where $\pi $ is the L\'{e}vy measure of the driving process $Z$. Assume that
there is a function $\phi (\sigma )$ such that $\lim_{\sigma \rightarrow
0}\phi (\sigma )=0$ and for $i=0,\ldots ,n_{T}-1,$%
\begin{equation}
\left\vert \mathbf{E}h(Z_{\tau _{i+1}}-Z_{\tau _{i}})-\mathbf{E}h(\zeta
_{i+1})\right\vert \leq C|h|_{\beta }\phi (\sigma )(\tau _{i+1}-\tau
_{i}),h\in \tilde{C}^{\beta }(\mathbf{R}^{d}),  \label{fo5}
\end{equation}%
for some constant $C$.

Then there is a constant $C$ (independent of $\sigma ,\delta $) such that
for all $g\in \tilde{C}^{\beta }(\mathbf{R}^{d})$%
\begin{equation*}
|\mathbf{E}g(\tilde{Y}_{T})-\mathbf{E}g(X_{T})|\leq C|g|_{\beta }[\delta ^{%
\frac{\beta }{\alpha }-1}+\phi (\sigma )].
\end{equation*}
\end{theorem}

The same way as Corollary \ref{co2} (see the proof below) we have the
following statement.

\begin{corollary}
Let assumptions of Theorem \ref{main2} hold and $\nu \in (0,\beta ]$. Then
there is a constant $C$ such that for all $g\in \tilde{C}^{\nu }(\mathbf{R}%
^{d})$ 
\begin{equation*}
|\mathbf{E}g(\tilde{Y}_{T})-\mathbf{E}g(X_{T})|\leq C|g|_{\nu }[\delta ^{\nu
(\frac{1}{\alpha }-\frac{1}{\beta })}+\phi (\sigma )^{\frac{\nu }{\beta }}].
\end{equation*}
\end{corollary}

\begin{remark}
(i) Assume the assumptions of Theorem \ref{main2} hold. Since $\lim_{\sigma
\rightarrow 0}\phi (\sigma )=0$, for each $\delta >0$ there is $\sigma
=\sigma (\delta )$ such that $\phi (\sigma (\delta ))\leq \delta ^{\frac{%
\beta }{\alpha }-1}$ and therefore%
\begin{equation*}
|\mathbf{E}g(\tilde{Y}_{T})-\mathbf{E}g(X_{T})|\leq C|g|_{\beta }\delta ^{%
\frac{\beta }{\alpha }-1}.
\end{equation*}%
In particular, if $\phi (\sigma )\leq C\sigma ^{\mu }$ with $\mu >0$ (it is
the case in Example \ref{ex1} for a small jumps $\alpha ^{\prime }$%
-stable-like driving process $Z$ with $\alpha ^{\prime }<\alpha $), then we
can choose $\sigma ^{\mu }=\delta ^{\frac{\beta }{\alpha }-1}$ or $\sigma
=\delta ^{(\frac{\beta }{\alpha }-1)\mu ^{-1}}.$

(ii) In order to study precisely the case of unbounded test functions (like
one in \cite{jkmp}), one would need to solve first the backward Kolmogorov
equation in H\"{o}lder spaces with weights that are defined by the powers of 
$w(x)=\left( 1+|x|^{2}\right) ^{1/2},x\in \mathbf{R}^{d}$.
\end{remark}

Applying Theorem \ref{main2} to the model of Example \ref{ex1} we have

\begin{proposition}
\label{propa1}Let $\alpha <\beta \leq \mu \leq 2\alpha $, and let $%
a^{i},b^{ij}\in \tilde{C}^{\beta }(\mathbf{R}^{d})$, $G^{ij}\in \tilde{C}%
^{\beta \vee 1}(\mathbf{R}^{d}),$ and 
\begin{equation*}
\int_{|\upsilon |\leq 1}|\upsilon |^{\alpha }d\pi +\int_{\left\vert \upsilon
\right\vert \geq 1}|\upsilon |^{\mu }\pi (d\upsilon )<\infty ,
\end{equation*}%
where $\pi $ is the L\'{e}vy measure of the driving process $Z$. For the
approximate Euler scheme in Example \ref{ex1}, there is a constant $C$
(independent of $\sigma ,\delta $) such that for all $g\in \tilde{C}^{\beta
}(\mathbf{R}^{d})$%
\begin{equation*}
|\mathbf{E}g(\tilde{Y}_{T})-\mathbf{E}g(X_{T})|\leq C|g|_{\beta }[\delta ^{%
\frac{\beta }{\alpha }-1}+\int_{|\upsilon |\leq \sigma }|\upsilon |^{\beta
\wedge 3}d\pi ].
\end{equation*}
\end{proposition}

\subsection{Approximate jump-adapted Euler scheme}

As in Example \ref{ex1}, for $\sigma \in (0,1)$ we approximate the
increments of the driving process 
\begin{equation*}
Z_{t}=\int_{0}^{t}\int (1-\chi _{\alpha }(\upsilon ))\upsilon p(ds,d\upsilon
)+\int_{0}^{t}\int \chi _{\alpha }(\upsilon )\upsilon q(ds,d\upsilon ),t\in
\lbrack 0,T],
\end{equation*}%
by the increments of%
\begin{equation*}
\tilde{Z}_{t}=Z_{t}^{\sigma }+R_{t}^{\sigma },
\end{equation*}%
with%
\begin{equation*}
Z_{t}^{\sigma }=\int_{0}^{t}\int_{|\upsilon |>\sigma }(1-\chi _{\alpha
}(\upsilon ))\upsilon p(ds,d\upsilon )+\int_{0}^{t}\int_{|\upsilon |>\sigma
}\chi _{\alpha }(\upsilon )\upsilon q(ds,d\upsilon )
\end{equation*}%
and 
\begin{equation*}
R_{t}^{\sigma }=\left\{ 
\begin{array}{ll}
t\int_{|\upsilon |\leq \sigma }\upsilon \pi (d\upsilon ) & \text{if }\alpha
<\beta \in (1,2],\alpha \in (0,1], \\ 
&  \\ 
B^{\sigma }\tilde{W}_{t} & \text{if }\alpha <\beta \in (2,4],\alpha \in
(1,2], \\ 
&  \\ 
0 & \text{otherwise,}%
\end{array}%
\right.
\end{equation*}%
where $B^{\sigma }$ is the square root of the positive definite $m\times m\,$%
-matrix $\left( \int_{|\upsilon |\leq \sigma }\upsilon _{i}\upsilon _{j}d\pi
\right) _{1\leq i,j\leq m}$ and $\tilde{W}_{t}$ is a standard independent
Wiener process in $\mathbf{R}^{m}$.

Given $\sigma \in (0,1),\delta >0,$ consider the following $Z^{\sigma }$%
-jump adapted time discretization (see \cite{MiP88}): $\tau _{0}=0,$%
\begin{equation}
\tau _{i+1}=\inf \left( t>\tau _{i}:\Delta Z_{t}^{\sigma }\neq 0\right)
\wedge (\tau _{i}+\delta )\wedge T.  \label{apf1}
\end{equation}%
In this case the time discretization $\{\tau _{i},i=0,\ldots ,n_{T}\}$ of
the interval $[0,T]$ is random, $\tau _{i}$ are stopping times. We
approximate $X_{t}$ by%
\begin{equation}
\hat{Y}_{t}=X_{0}+\int_{0}^{t}a(\hat{Y}_{\tau _{i_{s}}})ds+\int_{0}^{t}b(%
\hat{Y}_{\tau _{i_{s}}})dW_{s}+\int_{0}^{t}G(\hat{Y}_{\tau _{i_{s}}})d\tilde{%
Z}_{s},t\in \lbrack 0,T].  \label{apf2}
\end{equation}

The following error estimate holds.

\begin{theorem}
\label{propa2}Let $\alpha <\beta \leq \mu \leq 2\alpha $, and let $%
a^{i},b^{ij}\in \tilde{C}^{\beta }(\mathbf{R}^{d})$, $G^{ij}\in \tilde{C}%
^{\beta \vee 1}(\mathbf{R}^{d}),$ and 
\begin{equation*}
\int_{|\upsilon |\leq 1}|\upsilon |^{\alpha }d\pi +\int_{\left\vert \upsilon
\right\vert \geq 1}|\upsilon |^{\mu }\pi (d\upsilon )<\infty ,
\end{equation*}%
where $\pi $ is the L\'{e}vy measure of the driving process $Z$.

Then there is a constant $C$ (independent of $\sigma ,\delta $) such that
for all $g\in \tilde{C}^{\beta }(\mathbf{R}^{d})$%
\begin{equation*}
|\mathbf{E}g(\hat{Y}_{T})-\mathbf{E}g(X_{T})|\leq C|g|_{\beta }[\{(\delta
\wedge \lambda _{\sigma }^{-1})\tilde{\lambda}_{\sigma }\}^{\frac{\beta }{%
\alpha }-1}+\int_{|\upsilon |\leq \sigma }|\upsilon |^{\beta \wedge 3}d\pi ],
\end{equation*}%
where $\lambda _{\sigma }=\pi \left( \left\{ |\upsilon |>\sigma \right\}
\right) $ and 
\begin{equation*}
\tilde{\lambda}_{\sigma }=1+1_{\alpha \in (1,2)}|\int_{\sigma <|\upsilon
|\leq 1}\upsilon d\pi |.
\end{equation*}
\end{theorem}

In particular, the following statement holds.

\begin{corollary}
\label{co4}Suppose the assumptions of Theorem \ref{propa2} hold.

(i) If $\delta =T$ (only jump moments are chosen for the time
discretization), then%
\begin{equation*}
|\mathbf{E}g(\hat{Y}_{T})-\mathbf{E}g(X_{T})|\leq C|g|_{\beta }[\left( \frac{%
\tilde{\lambda}_{\sigma }}{\lambda _{\sigma }}\right) {}^{\frac{\beta }{%
\alpha }-1}+\int_{|\upsilon |\leq \sigma }|\upsilon |^{\beta \wedge 3}d\pi ].
\end{equation*}

(ii) If $\sup_{\sigma \in (0,1)}|\int_{\sigma <|\upsilon |\leq 1}\upsilon
d\pi |<\infty $ for $\alpha \in (1,2)$, then 
\begin{equation*}
|\mathbf{E}g(\hat{Y}_{T})-\mathbf{E}g(X_{T})|\leq C|g|_{\beta }[(\delta
\wedge \lambda _{\sigma }^{-1}){}^{\frac{\beta }{\alpha }-1}+\int_{|\upsilon
|\leq \sigma }|\upsilon |^{\beta \wedge 3}d\pi ],
\end{equation*}%
where $\lambda _{\sigma }=\pi \left( \left\{ |\upsilon |>\sigma \right\}
\right) .$
\end{corollary}

\subsection{Outline of Proof of Theorem \protect\ref{thm:main}}

\label{sec:outline}

To prove Theorem~\ref{thm:main}, as in \cite{MiP911} and \cite{mikz3}, the
solution to the backward Kolmogorov equation associated with $X_{t}$ is
used. First we introduce the operator of the Kolmogorov equation associated
with $X_{t}$.

For $u\in \tilde{C}^{\beta }(H),\beta >\alpha $, denote 
\begin{eqnarray*}
L_{z}u(t,x) &=&(a(z),\nabla _{x}u(t,x))+\frac{1}{2}%
\sum_{i,j=1}^{d}(b^{i}(z),b^{j}(z))\partial _{ij}^{2}u(x) \\
&&+\int_{\mathbf{R}_{0}^{m}}\big[u(t,x+G(z)\upsilon )-u(t,x)-\chi _{\alpha
}(\upsilon )(\nabla _{x}u(t,x),G(z)\upsilon ))\big]\pi (d\upsilon ), \\
Lu(t,x) &=&L_{x}u(t,x)=L_{z}u(t,x)|_{z=x},
\end{eqnarray*}%
where $b^{i}(z)=(b^{ij}(z))_{1\leq j\leq m},i=1,\ldots ,d.$

\begin{remark}
\label{lrenew1}Under assumptions of Theorem \ref{thm:main}, there exists a
unique strong solution to equation \textup{(\ref{one})} and the stochastic
process 
\begin{equation*}
u(X_{t})-\int_{0}^{t}Lu(X_{s})ds,\forall u\in \tilde{C}^{\beta }(\mathbf{R}%
^{d})
\end{equation*}%
with $\beta >\alpha $ is a martingale. The operator $L$ is the generator of $%
X_{t}$ defined in \textup{(\ref{one})}.
\end{remark}

If $v(t,x),(t,x)\in H$ satisfies the backward Kolmogorov equation 
\begin{eqnarray*}
\big(\partial _{t}+L\big)v(t,x) &=&0,\quad 0\leq t\leq T, \\
v(T,x) &=&g(x),
\end{eqnarray*}%
then by It\^{o}'s formula 
\begin{equation*}
\mathbf{E}[g(Y_{T})]-\mathbf{E}[g(X_{T})]=\mathbf{E}[v(T,Y_{T})-v(0,Y_{0})]=%
\mathbf{E}\big[\int_{0}^{T}(\partial _{t}+L_{Y_{\tau _{i_{s}}}})v(s,Y_{s})ds%
\big].
\end{equation*}%
The regularity of $v$ determines the one-step estimate and the rate of
convergence of the approximation.


\section{Backward Kolmogorov Equation}


In Lipshitz spaces $\tilde{C}^{\beta }(H),$ consider the backward Kolmogorov
equation associated with $X_{t}$: 
\begin{eqnarray}
\big(\partial _{t}+L\big)u(t,x) &=&f(t,x),  \label{eq1} \\
u(T,x) &=&g(x).  \notag
\end{eqnarray}

\begin{definition}
\label{def1}Let $f,g$ be measurable and bounded functions. We say that $u\in 
\tilde{C}^{\beta }(H)$ with $\beta >\alpha $ is a solution to $(\ref{eq1})$
if 
\begin{equation}
u(t,x)=g(x)+\int_{t}^{T}\big[Lu(s,x)-f(s,x)\big]ds,\forall (t,x)\in H.
\label{defs}
\end{equation}
\end{definition}


First we show that $L$:$\tilde{C}^{\beta }(H)\rightarrow $ $\tilde{C}^{\beta
-\alpha }(H)$ is continuous.

\begin{lemma}
\label{le3}Let $\alpha <\beta \leq \mu \leq 2\alpha ,$%
\begin{equation*}
\int_{|\upsilon |\leq 1}|\upsilon |^{\alpha }d\pi +\int_{|\upsilon
|>1}|\upsilon |^{\mu }d\pi <\infty
\end{equation*}%
and $a^{i},b^{ij}\in \tilde{C}^{\beta }(\mathbf{R}^{d}),G^{ij}\in \tilde{C}%
^{\beta \vee 1}(\mathbf{R}^{d})$. Then for any $v\in \tilde{C}^{\beta }(%
\mathbf{R}^{d})$ we have $Lv\in \tilde{C}^{\beta -\alpha }(\mathbf{R}^{d})$
and there is a constant independent of $v$ such that%
\begin{equation*}
|Lv|_{\beta -\alpha }\leq C|v|_{\beta }.
\end{equation*}
\end{lemma}

\begin{proof}
Let 
\begin{equation*}
Bv(x)=\int \left[ v(x+G\left( x\right) \upsilon )-v(x)-\chi _{\alpha
}(\upsilon )(\nabla v(x),G(x)\upsilon )\right] d\pi \text{.}
\end{equation*}%
Then 
\begin{equation*}
Lv=Bv+(a(x),\nabla v(x))+\frac{1}{2}(b^{i}(x),b^{j}(x))\partial
_{ij}^{2}v(x).
\end{equation*}%
By Proposition 13 in \cite{mikz3}, $Bv\in \tilde{C}^{\beta -\alpha }(\mathbf{%
R}^{d})$ if $\beta -\alpha \notin \mathbf{N}$ and $|Bv|_{\beta -\alpha }\leq
C|v|_{\beta }.$ In this case, obviously, $Lv\in \tilde{C}^{\beta -\alpha }(%
\mathbf{R}^{d})$ as well.

If $\alpha >1,\beta =1+\alpha ,$ then%
\begin{eqnarray*}
Bv(x) &=&\int_{|\upsilon |\leq 1}\int_{0}^{1}[\nabla v(x+sG(x)\upsilon
)-\nabla v(x)]G(x)\upsilon ]dsd\pi \\
&&+\int_{|\upsilon |>1}\left[ v(x+G(x)\upsilon )-v(x)\right] d\pi .
\end{eqnarray*}%
Since%
\begin{eqnarray*}
\nabla (Bv(x)) &=&\int_{|\upsilon |\leq 1}\int_{0}^{1}[\partial
^{2}v(x+sG(x)\upsilon )-\partial ^{2}v(x)]G(x)\upsilon ]dsd\pi \\
&&+\int_{|\upsilon |\leq 1}\int_{0}^{1}\partial ^{2}v(x+sG(x)\upsilon
)\nabla G(x)\upsilon G(x)\upsilon dsd\pi \\
&&+\int_{|\upsilon |>1}\left[ \nabla v(x+G(x)\upsilon )-\nabla v(x)\right]
d\pi \\
&&+\int_{|\upsilon |>1}\nabla v(x+G(x)\upsilon )\nabla G(x)\upsilon d\pi ,
\end{eqnarray*}%
it follows that $\sup_{x}|\nabla (Bv(x))|\leq C|v|_{\beta }.$ Therefore $%
|Lv|_{\beta -\alpha }\leq C|v|_{\beta }$ as well$.$ If $\alpha =1$ and $%
\beta =2$, then%
\begin{eqnarray*}
|\nabla Bv(x)| &=&\int [\nabla v(x+G(x)\upsilon )-\nabla v(x)]d\pi +\int
\nabla v(x+G(x)\upsilon )G(x)\upsilon d\pi , \\
\sup_{x}|\nabla Bv(x)| &\leq &C|v|_{\beta }
\end{eqnarray*}%
and $|Lv|_{\beta -\alpha }\leq C|v|_{\beta }$. The case $\beta =4,\alpha =2$
is considered in a similar way.
\end{proof}

The main result of this section is the following statement.

\begin{theorem}
\label{thm:StoCP}Let $\alpha <\beta \leq \mu \leq 2\alpha $, and 
\begin{equation*}
\int_{|\upsilon |\leq 1}|\upsilon |^{\alpha }\pi (d\upsilon
)+\int_{|\upsilon |>1}|\upsilon |^{\mu }\pi (d\upsilon )<\infty .
\end{equation*}%
Assume $a^{i},b^{ij}\in \tilde{C}^{\beta }(\mathbf{R}^{d}),G^{ij}\in \tilde{C%
}^{\beta \vee 1}(\mathbf{R}^{d})$. Then for each $f\in \tilde{C}^{\beta }(%
\mathbf{R}^{d}),g\in \tilde{C}^{\beta }(\mathbf{R}^{d})$, there exists a
unique solution $u\in \tilde{C}^{\beta }(H)$ to \textup{(\ref{eq1})} and a
constant $C$ independent of $f,g$ such that $|u|_{\beta }\leq C(|f|_{\beta
}+|g|_{\beta }).$
\end{theorem}

To prove Theorem \ref{thm:StoCP}, for $\varepsilon \in (0,1)$ we consider a
non-degenerate equation%
\begin{eqnarray}
\big(\partial _{t}+L^{\varepsilon }\big)u(t,x) &=&f(t,x),  \label{eq2} \\
u(T,x) &=&g_{\varepsilon }(x),  \notag
\end{eqnarray}%
where $L^{\varepsilon }u=-\varepsilon ^{\alpha }\left( -\Delta \right)
^{\alpha /2}u+Lu$ and 
\begin{equation*}
g_{\varepsilon }(x)=\int g(y)w^{\varepsilon }(x-y)dy=\int
g(x-y)w^{\varepsilon }(y)dy,x\in \mathbf{R}^{d}
\end{equation*}%
with $w^{\varepsilon }(x)=\varepsilon ^{-d}w\left( \frac{x}{\varepsilon }%
\right) ,x\in \mathbf{R}^{d},$ $w\in C_{0}^{\infty }(\mathbf{R}^{d}),\int
wdx=1.$

An obvious consequence of Corollary 9 in \cite{mikz3} is the following
statement.

\begin{lemma}
\label{lem1}(see Corollary 9 in \cite{mikz3}) Let $\alpha <\beta \leq \mu
\leq 2\alpha $,%
\begin{equation*}
\int_{|\upsilon |\leq 1}|\upsilon |^{\alpha }\pi (d\upsilon
)+\int_{|\upsilon |>1}|\upsilon |^{\mu }\pi (d\upsilon )<\infty ,
\end{equation*}%
and $a^{i},b^{ij},g,f\in \tilde{C}^{\beta }(\mathbf{R}^{d}),G^{ij}\in \tilde{%
C}^{\beta \vee 1}(\mathbf{R}^{d})$. Then for each $\varepsilon \in (0,1)$
there is $\bar{\beta}>2\alpha $ and a unique $u=u_{\varepsilon }\in \tilde{C}%
^{\bar{\beta}}(H)$ solving (\ref{eq2}).
\end{lemma}

We separate in the operator $L^{\varepsilon }$ its "bounded jump" part $\bar{%
L}^{\varepsilon }v(x)=\bar{L}_{z}^{\varepsilon }v(x)|_{z=x}$ with%
\begin{eqnarray*}
\bar{L}_{z}^{\varepsilon }v(x) &=&-\varepsilon ^{\alpha }\left( -\Delta
\right) ^{\alpha /2}u+(a(z),\nabla _{x}v(x)) \\
&&+\frac{1}{2}\sum_{i,j=1}^{d}(b^{i}(z),b^{j}(z))\partial _{ij}^{2}v(x) \\
&&+\int_{|\upsilon |\leq 1}\left[ v(x+G\left( z\right) \upsilon )-v(x)-\chi
_{\alpha }(\upsilon )(\nabla v(x),G(z)\upsilon )\right] d\pi ,
\end{eqnarray*}%
$z,x\in R^{d},v\in C_{0}^{\infty }(\mathbf{R}^{d})$, so that%
\begin{equation*}
L_{z}^{\varepsilon }v(x)=\bar{L}_{z}^{\varepsilon }v(x)+\int_{|\upsilon
|>1}[v(x+G(z)y)-v(x)]d\pi ,x,z\in \mathbf{R}^{d}.
\end{equation*}

\begin{remark}
\label{re3}If the assumptions of Lemma \ref{lem1} hold and $u_{\varepsilon
}\in \tilde{C}^{\bar{\beta}}(H)$ solves (\ref{eq2}) with $\bar{\beta}%
>2\alpha $, then $u_{\varepsilon }$ satisfies the following equation as well:%
\begin{eqnarray}
\big(\partial _{t}+\bar{L}^{\varepsilon }\big)u(t,x) &=&F(u,t,x),
\label{eq3} \\
u(T,x) &=&g_{\varepsilon }(x),  \notag
\end{eqnarray}%
where $F(u,t,x)=F_{z}(u,t,x)|_{z=x}$ with 
\begin{equation*}
F_{z}(u,t,x)=f(t,x)-\int_{|\upsilon |>1}\big[u(t,x+G(z)\upsilon )-u(t,x)\big]%
d\pi .
\end{equation*}
\end{remark}

Using a probabilistic form of a maximum principle we will derive uniform
(independent of $\varepsilon $) $\tilde{C}^{\beta }$-norm estimates of $%
u_{\varepsilon }$ and passing to the limit as $\varepsilon \rightarrow 0$ we
will obtain $u\in \tilde{C}^{\beta }(H)$ solving (\ref{eq1}). First we prove
some auxiliary statements.

Let 
\begin{eqnarray}
\tilde{Z}_{t} &=&\int_{0}^{t}\int_{|\upsilon |\leq 1}[(1-\chi _{\alpha
}(\upsilon ))\upsilon p(dt,d\upsilon )+\chi _{\alpha }(\upsilon )\upsilon
q(dt,d\upsilon )]  \notag \\
&=&\int_{0}^{t}\int_{|v|\leq 1}vq(dt,d\upsilon )+t\int_{|\upsilon |\leq
1}(1-\chi _{\alpha }(\upsilon ))\upsilon d\pi .  \label{3}
\end{eqnarray}%
For $(s,x)\in H,h\in \mathbf{R}^{d},\xi \in \mathbf{R}^{d}$, the following
stochastic processes in $[s,T]$ are used to derive the uniform estimates:%
\begin{eqnarray}
dU_{t} &=&\varepsilon dZ_{t}^{\alpha }+a(U_{t})dt+b(U_{t})dW_{t}+G(U_{t-})d%
\tilde{Z}_{t},  \notag \\
&&  \notag \\
dH_{t} &=&[a(U_{t}+H_{t})-a(U_{t})]dt+[b(U_{t}+H_{t})-b(U_{t})]dW_{t}  \notag
\\
&&+\left[ G(U_{t-}+H_{t-})-G(U_{t-})\right] d\tilde{Z}_{t},  \label{4} \\
&&  \notag \\
d\bar{V}_{t} &=&a^{(1)}(U_{t}+H_{t};\bar{V}_{t})dt+b^{(1)}(U_{t}+H_{t};\bar{V%
}_{t})dW_{t}  \notag \\
&&+\int_{|\upsilon |\leq 1}G^{(1)}(U_{t-}+H_{t-};\bar{V}_{t-})d\tilde{Z}_{t},
\notag \\
&&  \notag \\
dV_{t}
&=&a^{(1)}(U_{t};V_{t})dt+b^{(1)}(U_{t};V_{t})dW_{t}+G^{(1)}(U_{t-};V_{t-})d%
\tilde{Z}_{t},  \notag \\
U_{s} &=&x,H_{s}=h,V_{s}=\xi ,\bar{V}_{s}=\xi \text{,}  \notag
\end{eqnarray}%
where $Z^{\alpha }$ is $\mathbf{R}^{d}$- valued spherically symmetric $%
\alpha $-stable process corresponding to $(-\Delta )^{\alpha /2}$ and
independent of $Z$. Recall for a function $v$ on $\mathbf{R}^{d}$ we denote $%
v^{(1)}(x;\xi )=(\nabla v(x),\xi ),x,\xi \in \mathbf{R}^{d}$ and, for
example, componentwise,%
\begin{equation*}
dV_{t}^{j}=(\nabla a^{j}(U_{t}),V_{t})dt+\sum_{i=1}^{n}(\nabla
b^{ji}(U_{t}),V_{t})dW_{t}^{i}+\sum_{i=1}^{m}(\nabla G^{ji}(U_{t-}),V_{t-})d%
\tilde{Z}_{t}^{i},
\end{equation*}%
$j=1,\ldots ,d.$

\begin{lemma}
\label{lem2}(a) If $a^{i},b_{j}^{i},G^{ij}\in \tilde{C}^{1}(\mathbf{R}^{d}),$
then for each \thinspace $l\geq 2$ there is a constant $C$ such that%
\begin{equation*}
\mathbf{E[}\sup_{s\leq t\leq T}|H_{t}|^{l}]\leq C|h|^{l}.
\end{equation*}

(b) If $a^{i},b_{j}^{i},G^{ij}\in \tilde{C}^{1+\kappa }(\mathbf{R}^{d})$
with $\kappa \in (0,1]$, then for each \thinspace $l\geq 2$ there is a
constant $C$ such that%
\begin{eqnarray*}
\mathbf{E[}\sup_{s\leq t\leq T}|V_{s}|^{l}+\sup_{s\leq t\leq T}|\bar{V}%
_{s}|^{l}] &\leq &C|\xi |^{l}, \\
\mathbf{E}[\sup_{s\leq t\leq T}|V_{t}-\bar{V}_{t}|^{l}] &\leq &C|\xi
|^{l}|h|^{l\kappa }.
\end{eqnarray*}
\end{lemma}

\begin{proof}
(a) Since (\ref{3}) holds, we have by H\"{o}lder inequality and martingale
moment estimates (see \cite{mikprag}, \cite{PrT97})%
\begin{eqnarray*}
\mathbf{E}\sup_{s\leq r\leq t}|H_{r}|^{l} &\leq &C[|h|^{l}+\mathbf{E[}\left(
\int_{s}^{t}|H_{r}|^{2}dr\right) ^{l/2}]+\mathbf{E[}%
\int_{s}^{t}|H_{r}|^{l}dr] \\
&\leq &C[|h|^{l}+\mathbf{E[}\int_{s}^{t}\sup_{s\leq r^{\prime }\leq
r}|H_{r^{\prime }}|^{l}dr],s\leq t\leq T.
\end{eqnarray*}%
and inequality follows by Gronwall lemma.

(b) Similarly, for each $l\geq 2$, there is a constant $C$ so that%
\begin{equation*}
\mathbf{E[}\sup_{s\leq t\leq T}|V_{s}|^{l}+\sup_{s\leq t\leq T}|\bar{V}%
_{s}|^{l}]\leq C|\xi |^{l}.
\end{equation*}%
Then%
\begin{eqnarray*}
\mathbf{E}\sup_{s\leq r\leq t}|V_{r}-\bar{V}_{r}|^{l} &\leq &C[\mathbf{E[}%
\left( \int_{s}^{t}|H_{r}|^{2\kappa }|\bar{V}_{r}|^{2}dr\right) ^{l/2}]+%
\mathbf{E[}\int_{s}^{t}|H_{r}|^{\kappa l}|\bar{V}_{r}|^{l}dr] \\
&&+\mathbf{E[}\left( \int_{s}^{t}|\bar{V}_{r}-V_{r}|^{2}dr\right) ^{l/2}]+%
\mathbf{E[}\int_{s}^{t}|\bar{V}_{r}-V_{r}|^{l}dr] \\
&\leq &C[\mathbf{E}\int_{s}^{t}|H_{r}|^{\kappa l}|\bar{V}_{r}|^{l}dr+\mathbf{%
E}\int_{s}^{t}|\bar{V}_{r}-V_{r}|^{l}dr],s\leq t\leq T.
\end{eqnarray*}%
By Gronwall lemma,%
\begin{eqnarray*}
\mathbf{E}\sup_{s\leq r\leq T}|V_{r}-\bar{V}_{r}|^{l} &\leq &C\mathbf{E}%
\int_{s}^{T}|H_{r}|^{\kappa l}|\bar{V}_{r}|^{l}dr \\
&\leq &C\int_{s}^{T}[\mathbf{E(}|H_{r}|^{2\kappa l})]^{1/2}[\mathbf{E}(|\bar{%
V}_{r}|^{2l})^{1/2}]dr \\
&\leq &C|\xi |^{l}|h|^{\kappa l}.
\end{eqnarray*}
\end{proof}

\subsection{Proof of Theorem \protect\ref{thm:StoCP}}

\emph{1. Existence.} By Lemma \ref{lem1}, for each $\varepsilon \in (0,1)$
there is a unique solution $u_{\varepsilon }\in \tilde{C}^{\bar{\beta}}(H)$
to (\ref{eq2}) for some $\bar{\beta}>2\alpha $. By Remark \ref{re3}, (\ref%
{eq3}) holds as well. Let $(s,x)\in H$ and $U_{t}$ solves (\ref{4})$.$ By It%
\^{o} formula,%
\begin{equation*}
\mathbf{E}g_{\varepsilon }(U_{T})-u_{\varepsilon }(s,x)=\mathbf{E}%
\int_{s}^{T}F(u_{\varepsilon },r,U_{r})dr
\end{equation*}%
and%
\begin{equation*}
|u_{\varepsilon }(s,\cdot )|_{0}\leq |g|_{0}+\int_{s}^{T}(|f(r,\cdot
)|_{0}+C|u_{\varepsilon }(r,\cdot )|_{0}dr.
\end{equation*}%
By Gronwall lemma, there is a constant not depending on $u_{\varepsilon }$
and $\varepsilon $ such that 
\begin{equation*}
\sup_{0\leq t\leq T}|u_{\varepsilon }(t,\cdot )|_{0}\leq
C[|g|_{0}+\int_{0}^{T}|f(r,\cdot )|_{0}dr].
\end{equation*}

As suggested in \cite{kry}, we estimate multilinear forms associated to the
derivatives of $u.$ Let $k=[\beta ]^{-},(t,x)\in H,\xi ^{1},\ldots ,\xi
^{k}\in \mathbf{R}^{d}$ and%
\begin{eqnarray*}
u_{\varepsilon }^{(k)}(t,x;\xi ^{1},\ldots ,\xi ^{k}) &=&\sum_{i_{1},\ldots
,i_{k}=1}^{d}\frac{\partial ^{k}u(t,x)}{\partial x_{i_{k}}\ldots x_{i_{1}}}%
\xi _{i_{1}}^{1}\ldots \xi _{i_{k}}^{k}\text{ if }k\geq 1, \\
u_{\varepsilon }^{(0)}(t,x) &=&u_{\varepsilon }(t,x)\text{.}
\end{eqnarray*}%
For $z\in \mathbf{R}^{d},(t,x)\in H,\xi ^{1}\in \mathbf{R}^{d},\ldots ,\xi
^{k}\in \mathbf{R}^{d},$ let%
\begin{eqnarray*}
&&\mathcal{P}_{z}u_{\varepsilon }^{(k)}(t,x;\xi ^{1},\ldots ,\xi
^{k})=-\varepsilon ^{\alpha }(-\Delta _{x})^{\alpha /2}u_{\varepsilon
}^{(k)}(t,x,\xi ^{1},\ldots ,\xi ^{k}) \\
&&+\int_{|\upsilon |\leq 1}\{u_{\varepsilon }^{(k)}(x+G(z)\upsilon ;\xi
^{1}+G^{(1)}(z;\xi ^{1})\upsilon ,\ldots ,\xi ^{k}+G_{\varepsilon
}^{(1)}(z;\xi ^{k})\upsilon )-u_{\varepsilon }^{(k)}(x;\xi ^{1},\ldots ,\xi
^{k}) \\
&&-\chi _{\alpha }(\upsilon )[(\nabla _{x}u_{\varepsilon }^{(k)}(x;\xi
_{1},\ldots ,\xi ^{k}),G(z)\upsilon )-\sum_{l=1}^{k}(\nabla _{\xi
^{l}}u_{\varepsilon }^{(k)}(x;\xi _{1},\ldots ,\xi ^{k}),G^{(1)}(z;\xi
^{l})\upsilon )]\}d\pi \\
&&+(a(z),\nabla _{x}u_{\varepsilon }^{(k)}(x;\xi ^{1},\ldots ,\xi
^{k}))+\sum_{l=1}^{k}(\nabla _{\xi ^{l}}u^{(k)}(x;\xi ^{1},\ldots ,\xi
^{k}),a^{(1)}(z;\xi ^{l})) \\
&&+\frac{1}{2}\sum_{i,j}\{(b^{i}(z),b^{j}(z))\partial
_{ij}^{2}u_{\varepsilon }^{(k)}(x;\xi ^{1},\ldots ,\xi ^{k}) \\
&&+\sum_{l=1}^{k}[(b^{i,(1)}(z;\xi ^{l}),b^{j}(z))\partial _{\xi
_{i}^{l}x_{j}}u_{\varepsilon }^{(k)}(x;,\xi ^{1},\ldots ,\xi ^{k}) \\
&&+(b^{i}(z),b^{j,(1)}(z;\xi ^{l}))\partial _{x_{i}\xi
_{j}^{l}}u^{(k)}(x;,\xi ^{1},\ldots ,\xi ^{k})]\}.
\end{eqnarray*}

Differentiating both sides of (\ref{eq3}) and multiplying by $\xi
_{i_{1}}^{1}\ldots \xi _{i_{k}}^{k}$ we see that $u_{\varepsilon
}^{(k)}(t,x;\xi ^{1},\ldots ,\xi ^{k})$ satisfies the equation 
\begin{eqnarray}
&&\partial _{t}u_{\varepsilon }^{(k)}(t,x,\xi ^{1},\ldots ,\xi ^{k})+%
\mathcal{P}^{\varepsilon }u_{\varepsilon }^{(k)}(t,x,\xi ^{1},\ldots ,\xi
^{k})  \label{6} \\
&=&A(u_{\varepsilon },t,x,\xi ^{1},\ldots ,\xi ^{k}),  \notag
\end{eqnarray}%
where 
\begin{equation*}
A(u_{\varepsilon },t,x,\xi ^{1},\ldots ,\xi ^{k})=B(u_{\varepsilon },t,x,\xi
^{1},\ldots ,\xi ^{k})+F^{(k)}(u_{\varepsilon },t,x;\xi ^{1},\ldots \xi ^{k})
\end{equation*}%
and $B(u_{\varepsilon },t,x,\xi ^{1},\ldots ,\xi ^{k})$ is a finite sum of
the terms of the form%
\begin{eqnarray*}
&&[\nabla _{x}u_{\varepsilon }^{(l)}(t,x+G(x)\upsilon ;\xi ^{i_{1}},\ldots
,\xi ^{i_{l}})-\nabla _{x}u^{(l)}(t,x+G(x)\upsilon ;\xi ^{i_{1}},\ldots ,\xi
^{i_{l}})]\times  \\
&&\times G^{(k-l)}(x;\xi ^{i_{l+1}},\ldots ,\xi ^{i_{k}})\upsilon  \\
&=&\int_{0}^{1}\partial ^{2}u_{x}^{(l)}(t,x+sG(x)\upsilon ;\xi
^{i_{1}},\ldots ,\xi ^{i_{l}})G(x)\upsilon dsG^{(k-l)}(x;\xi
^{i_{l+1}},\ldots ,\xi ^{i_{k}})\upsilon 
\end{eqnarray*}%
with $l\leq k-2$ and%
\begin{equation*}
u^{(l)}(t,x+G(x)\upsilon ;\xi ^{i_{1}},\ldots ,\xi
^{i_{l}})G^{(l_{1})}(x;\xi ^{i_{1}^{1}},\ldots ,\xi ^{i_{k_{1}}^{1}})\ldots
G^{(l_{m})}(x;\xi ^{i_{1}^{m}},\ldots ,\xi ^{j_{k_{m}}^{m}})
\end{equation*}%
with $m\geq 2,l\leq k,l+l_{1}+\ldots +l_{m}=k$ and $(\xi ^{i_{1}},\ldots
,\xi ^{i_{l}},\ldots ,\xi ^{i_{k_{m}}^{m}})$ being a permutation of $\xi
^{1},\ldots ,\xi ^{k}$. In any case, there is a constant $C$ independent of $%
\varepsilon $ and $u_{\varepsilon }$ so that for all $(t,x)\in \lbrack
0,T]\times \mathbf{R}^{d},\xi ^{i}\in \mathbf{R}^{d},$ 
\begin{eqnarray}
|A(u_{\varepsilon },t,x,\xi ^{1},\ldots ,\xi ^{k})| &\leq &C(|u_{\varepsilon
}(t,\cdot )|_{k}+|f(t,\cdot )|_{k})|\xi ^{1}|\ldots |\xi ^{k}|,  \label{7} \\
|A(u_{\varepsilon },t,\cdot ,\xi ^{1},\ldots ,\xi ^{k})|_{\beta -k} &\leq
&C(|u_{\varepsilon }(t,\cdot )|_{\beta }+|f(t,\cdot )|_{\beta })|\xi
^{1}|\ldots |\xi ^{k}|,  \notag
\end{eqnarray}%
and 
\begin{eqnarray}
&&|A(u_{\varepsilon },t,x,\bar{\xi}^{1},\ldots ,\bar{\xi}^{k})-A(u_{%
\varepsilon },t,x,\xi ^{1},\ldots ,\xi ^{k})|  \label{8} \\
&\leq &C(|f(t,\cdot )|_{k}+|u_{\varepsilon }(t,\cdot
)|_{k})\sum_{l=1}^{k}|\xi ^{1}|\ldots |\xi ^{l-1}||\bar{\xi}^{l}-\xi ^{l}||%
\bar{\xi}^{l+1}|\ldots |\bar{\xi}^{k}|).  \notag
\end{eqnarray}

On the other hand, for any $(s,x)\in H$ with the processes defined in (\ref%
{4}), it follows by It\^{o} formula,%
\begin{eqnarray*}
&&\mathbf{E[}u_{\varepsilon }^{(k)}(T,U_{T},V_{T}^{1},\ldots
,V_{T}^{k})-u_{\varepsilon }^{(k)}(s,x,\xi ^{1},\ldots ,\xi ^{k})] \\
&& \\
&=&\mathbf{E[}g_{\varepsilon }^{(k)}(U_{T},V_{T}^{1},\ldots
,V_{T}^{k})-u_{\varepsilon }^{(k)}(s,x,\xi ^{1},\ldots ,\xi ^{k})] \\
&=&\mathbf{E}\int_{s}^{T}[\partial _{t}u_{\varepsilon
}^{(k)}(t,U_{t}{},V_{t}^{1},\ldots ,V_{t}^{k})+\mathcal{P}%
_{U_{t}}^{\varepsilon }u_{\varepsilon }^{(k)}(t,U_{t},V_{t}^{1},\ldots
,V_{t}^{k})]dt \\
&=&\mathbf{E}\int_{s}^{T}A(u_{\varepsilon },t,U_{t},V_{t}^{1},\ldots
,V_{t}^{k})]dt
\end{eqnarray*}%
and%
\begin{eqnarray*}
&&\mathbf{E}[u_{\varepsilon }^{(k)}(T,U_{T}+H_{T},\bar{V}_{T}^{1},\ldots ,%
\bar{V}_{T}^{k})-u_{\varepsilon }^{(k)}(T,U_{T},V_{T}^{1},\ldots ,V_{T}^{k})]
\\
&& \\
&&-[u_{\varepsilon }^{(k)}(s,x+h,\xi ^{1},\ldots ,\xi ^{k})-u_{\varepsilon
}^{(k)}(s,x,\xi ^{1},\ldots ,\xi ^{k})] \\
&& \\
&=&\mathbf{E}[g^{(k)}(U_{T}+H_{T},\bar{V}_{T}^{1},\ldots ,\bar{V}%
_{T}^{k})-g^{(k)}(U_{T},V_{T}^{1},\ldots ,V_{T}^{k})] \\
&& \\
&&-[u_{\varepsilon }^{(k)}(s,x+h,\xi ^{1},\ldots ,\xi ^{k})-u_{\varepsilon
}^{(k)}(s,x,\xi ^{1},\ldots ,\xi ^{k})] \\
&=&\mathbf{E}\int_{s}^{T}\{[\partial _{t}u_{\varepsilon
}^{(k)}(t,U_{t}+H_{t},\bar{V}_{t}^{1},\ldots ,\bar{V}_{t}^{k})+\mathcal{P}%
_{U_{t}+H_{t}}^{\varepsilon }u_{\varepsilon }^{(k)}(t,U_{t}+H_{t},\bar{V}%
_{t}^{1},\ldots ,\bar{V}_{t}^{k})] \\
&&-[\partial _{t}u_{\varepsilon }^{(k)}(t,U_{t}{},V_{t}^{1},\ldots
,V_{t}^{k})+\mathcal{P}_{U_{t}}^{\varepsilon }u_{\varepsilon
}^{(k)}(t,U_{t},V_{t}^{1},\ldots ,V_{t}^{k})]\}dt \\
&=&\mathbf{E}\int_{s}^{T}[A(u_{\varepsilon },t,U_{t}+H_{t},\bar{V}%
_{t}^{1},\ldots ,\bar{V}_{t}^{k})-A(u_{\varepsilon
},t,U_{t},V_{t}^{1},\ldots ,V_{t}^{k})]dt
\end{eqnarray*}%
Since by (\ref{7}) 
\begin{equation*}
|A(u_{\varepsilon },t,U_{t},V_{t}^{1},\ldots ,V_{t}^{k})|\leq
C(|u_{\varepsilon }(t,\cdot )|_{k}+|f(t,\cdot )|_{k})|V_{t}^{1}|\ldots
|V_{t}^{k}|,
\end{equation*}%
it follows by Lemma \ref{lem2} and H\"{o}lder inequality,%
\begin{equation}
\mathbf{E}|A(u_{\varepsilon },t,U_{t},V_{t}^{1},\ldots ,V_{t}^{k})|\leq
C(|u_{\varepsilon }(t,\cdot )|_{k}+|f(t,\cdot )|_{k})|\xi ^{1}|\ldots |\xi
^{k}|.  \label{9}
\end{equation}%
Since%
\begin{eqnarray*}
&&|A(u_{\varepsilon },t,U_{t}+H_{t},\bar{V}_{t}^{1},\ldots ,\bar{V}%
_{t}^{k})-A(u_{\varepsilon },t,U_{t},V_{t}^{1},\ldots ,V_{t}^{k})| \\
&& \\
&\leq &|A(u_{\varepsilon },t,U_{t}+H_{t},\bar{V}_{t}^{1},\ldots ,\bar{V}%
_{t}^{k})-A(u_{\varepsilon },t,U_{t},\bar{V}_{t}^{1},\ldots ,\bar{V}%
_{t}^{k})| \\
&& \\
&&+|A(u_{\varepsilon },t,U_{t},\bar{V}_{t}^{1},\ldots ,\bar{V}%
_{t}^{k})-A(u_{\varepsilon },t,U_{t},V_{t}^{1},\ldots ,V_{t}^{k})| \\
&=&A_{1}+A_{2},
\end{eqnarray*}%
it follows by the estimates (\ref{7}), (\ref{8}) and Lemma \ref{lem2} that%
\begin{eqnarray*}
\mathbf{E}A_{1} &\leq &C(\mathbf{E}|H_{t}|^{2(\beta -k)})^{1/2}\left(
|f(t,\cdot )|_{\beta }+|u(t,\cdot )|_{\beta }\right)  \\
&\leq &C|h|^{\beta -k}\left( |f(t,\cdot )|_{\beta }+|u(t,\cdot )|_{\beta
}\right) .
\end{eqnarray*}%
and for $|h|\leq 1$%
\begin{eqnarray*}
\mathbf{E}A_{2} &\leq &C(|f(t,\cdot )|_{k}+|u(t,\cdot )|_{k})\sum_{l=1}^{k}%
\mathbf{E}|V_{t}^{1}|\ldots |V_{t}^{l-1}||\bar{V}_{t}^{l}-V_{t}^{l}||\bar{V}%
_{t}^{l+1}|\ldots |\bar{V}_{t}^{k}|) \\
&\leq &C(|f(t,\cdot )|_{k}+|u(t,\cdot )|_{k})\sum_{l}(\mathbf{E}[|\bar{V}%
_{t}^{l}-V_{t}^{l}|^{2}])^{1/2}|\xi ^{1}|\ldots |\xi ^{l-1}||\xi
^{l+1}|\ldots |\xi ^{k}| \\
&\leq &C(|f(t,\cdot )|_{k}+|u(t,\cdot )|_{k})|\xi ^{1}|\ldots |\xi
^{k}||h|^{\beta -k}.
\end{eqnarray*}

Similarly, we estimate%
\begin{equation*}
\mathbf{E|}g_{\varepsilon }^{(k)}(U_{T},V_{T}^{1},\ldots ,V_{T}^{k})|\leq
C|g|_{k}|\xi ^{1}|\ldots |\xi ^{k}|
\end{equation*}%
and for $|h|\leq 1$%
\begin{eqnarray*}
&&\mathbf{E}|g_{\varepsilon }^{(k)}(U_{T}+H_{T},\bar{V}_{T}^{1},\ldots ,\bar{%
V}_{T}^{k})-g_{\varepsilon }^{(k)}(U_{T},V_{T}^{1},\ldots ,V_{T}^{k})| \\
&\leq &C|g|_{\beta }|h|^{\beta -k}|\xi ^{1}|\ldots |\xi ^{k}|.
\end{eqnarray*}

So, 
\begin{eqnarray*}
&&|u_{\varepsilon }^{(k)}(s,x;\xi ^{1},\ldots ,\xi ^{k})|_{0} \\
&\leq &C|\xi ^{1}|\ldots |\xi ^{k}|[|g|_{k}+\int_{s}^{T}(|u_{\varepsilon
}(t,\cdot )|_{k}+|f(t,\cdot )|_{k})dt],0\leq s\leq T,
\end{eqnarray*}%
and by Gronwall lemma,%
\begin{equation*}
\sup_{0\leq s\leq T}|u_{\varepsilon }^{(k)}(s,x;\xi ^{1},\ldots ,\xi
^{k})|_{0}\leq C|\xi ^{1}|\ldots |\xi ^{k}|[|g|_{k}+\int_{0}^{T}|f(t,\cdot
)|_{k})dt].
\end{equation*}%
Also, for $|h|\leq 1,x\in \mathbf{R}^{d},0\leq s\leq T,$%
\begin{eqnarray*}
&&|u_{\varepsilon }^{(k)}(s,x+h,\xi ^{1},\ldots ,\xi ^{k})-u_{\varepsilon
}^{(k)}(s,x,\xi ^{1},\ldots ,\xi ^{k})| \\
&\leq &C|h|^{\beta -k}|\xi ^{1}|\ldots |\xi ^{k}|[|g|_{\beta
}+\int_{s}^{T}(|f(t,\cdot )|_{\beta }+|u(t,\cdot )|_{\beta })dt],
\end{eqnarray*}%
and by Gronwall lemma,%
\begin{eqnarray*}
&&\sup_{0\leq s\leq T}|u^{(k)}(s,\cdot ,\xi ^{1},\ldots ,\xi ^{k})|_{\beta
-k} \\
&\leq &C|\xi ^{1}|\ldots |\xi ^{k}|[|g|_{\beta }+\int_{0}^{T}|f(t,\cdot
)|_{\beta }dt]
\end{eqnarray*}%
Therefore for each $\beta \in (\alpha ,2\alpha ],$%
\begin{equation}
\sup_{\varepsilon \in (0,1)}|u_{\varepsilon }|_{\beta }\leq C[|g|_{\beta
}+\int_{0}^{T}|f(t,\cdot )|_{\beta }dt].  \label{10}
\end{equation}%
Since for each $(s,x)\in H,$%
\begin{equation}
u_{\varepsilon }(s,x)=g_{\varepsilon }(x)+\int_{s}^{T}[L^{\varepsilon
}u_{\varepsilon }(t,x)-f(t,x)]dt,  \label{11}
\end{equation}%
and there is a constant $C>0$ so that for all $(t,x)\in H,h\in \mathbf{R}%
^{d},$ 
\begin{eqnarray}
&&|\partial _{t}u_{\varepsilon }(t,x+h)-\partial _{t}u_{\varepsilon }(t,x)|
\label{12} \\
&\leq &|L_{x+h}^{\varepsilon }u(t,x+h)-L^{\varepsilon }u_{\varepsilon
}u(t,x)|+|f(t,x+h)-f(t,x)|  \notag \\
&\leq &C|h|^{\tilde{\beta}-\alpha }(|u_{\varepsilon }|_{\tilde{\beta}%
}+|f|_{\beta })  \notag
\end{eqnarray}%
for some $\tilde{\beta}\in (\alpha ,\alpha +\alpha \wedge 1).$ It follows
from (\ref{10}) and (\ref{12}) that there is a sequence $\varepsilon
_{n}\rightarrow 0$ and $u\in \tilde{C}^{\beta }(H)$ such that such $%
u_{\varepsilon _{n}}\rightarrow u$ uniformly on compact sets of $H$. By (\ref%
{10}), $L^{\varepsilon }u_{\varepsilon }(t,x)\rightarrow Lu(t,x)$ pointwise
and passing to the limit in (\ref{11}), we see that $u\in \tilde{C}^{\beta
}(H)$ is a solution to (\ref{eq1}).

\emph{2. Uniqueness.} Let $u^{1},u^{2}\in \tilde{C}^{\beta }(H)$ be two
solutions to (\ref{eq1}). Then $v=u^{1}-u^{2}$ satisfies (\ref{eq1}) with $%
g=0,f=0$. Let $X_{t}^{s,x}$ be the solution to (\ref{one}) starting from $%
x\in \mathbf{R}^{d}$ at time moment $s$. Then by It\^{o} formula, 
\begin{eqnarray*}
-v(s,x) &=&\mathbf{E}v(T,X_{T}^{s,x})-v(s,x) \\
&=&\mathbf{E}\int_{s}^{T}\left[ \partial
_{t}v(r,X_{r}^{s,x})+Lv(r,X_{r}^{s,x})\right] dr=0
\end{eqnarray*}
and uniqueness follows.

\section{One-Step Estimate and Proof of Main Results\newline
}

First, we modify the mollified function estimates for the Lipshitz spaces.
Let $w\in C_{0}^{\infty }(\mathbf{R}^{d}),$ be a nonnegative smooth function
with support in $\{|x|\leq 1\}$ such that $w(x)=w(|x|)$, $x\in \mathbf{R}%
^{d},$ and $\int w(x)dx=1.$ Due to the symmetry, 
\begin{equation}
\int_{\mathbf{R}^{d}}x^{i}w(x)dx=0,i=1,\ldots ,d.  \label{ff26}
\end{equation}%
For $x\in \mathbf{R}^{d}$ and $\varepsilon \in (0,1)$, define $%
w^{\varepsilon }(x)=\varepsilon ^{-d}w\left( \frac{x}{\varepsilon }\right) $
and the convolution 
\begin{equation}
f^{\varepsilon }(x)=\int f(y)w^{\varepsilon }(x-y)dy=\int
f(x-y)w^{\varepsilon }(y)dy,x\in \mathbf{R}^{d}.  \label{maf7}
\end{equation}

\begin{lemma}
\label{lemn1}Let $\alpha <\beta \leq 2\alpha ,f\in \tilde{C}^{\beta -\alpha
}(\mathbf{R}^{d})$. Then%
\begin{equation}
|f^{\varepsilon }(x)-f(x)|\leq C\varepsilon ^{\beta -\alpha }|f|_{\beta
-\alpha },x\in \mathbf{R}^{d},  \label{ma8}
\end{equation}%
and there is a constant $C$ such that%
\begin{equation}
|Lf^{\varepsilon }|\leq C\varepsilon ^{\beta -2\alpha }|f|_{\beta -\alpha }
\label{ma9}
\end{equation}
\end{lemma}

\begin{proof}
Indeed, if $\beta -\alpha \leq 1$, then%
\begin{eqnarray*}
|f^{\varepsilon }(x)-f(x)| &\leq &\int |f(x-y)-f(x)|w^{\varepsilon }(y)dy, \\
&\leq &C|f|_{\beta -\alpha }\varepsilon ^{\beta -\alpha }.
\end{eqnarray*}%
If $\beta -\alpha \in (1,2],$ then 
\begin{eqnarray*}
|f^{\varepsilon }(x)-f(x)| &=&|\int f(x+y)-f(x)-(\nabla
f(x),y)w^{\varepsilon }(y)dy| \\
&\leq &\int \int_{0}^{1}|(\nabla f(x+sy)-\nabla f(x),y)|dsw^{\varepsilon
}(y)dydy \\
&\leq &C\varepsilon ^{\beta -\alpha }|f|_{\beta -\alpha }.
\end{eqnarray*}

According to Lemma 17 (iii) and Corollary 18 in \cite{mikz3}, for each $%
\beta $ so that $\beta -\alpha <\alpha $ 
\begin{equation*}
|Lf^{\varepsilon }|\leq C\varepsilon ^{(\beta -\alpha )-\alpha }|f|_{\beta
-\alpha }=C\varepsilon ^{\beta -2\alpha }|f|_{\beta -\alpha }.
\end{equation*}%
The inequality (\ref{ma9}) still holds for $\beta -\alpha =\alpha $ or $%
\beta =2\alpha $ by a straightforward estimate.
\end{proof}

We modify one-step estimate in \cite{mikz3} for Lipschitz spaces as well.

\begin{lemma}
\label{lem:expect}Let $\alpha <\beta \leq \mu \leq 2\alpha ,$%
\begin{equation*}
\int_{|\upsilon |\leq 1}|\upsilon |^{\alpha }d\pi +\int_{|\upsilon
|>1}|\upsilon |^{\mu }d\pi <\infty ,
\end{equation*}
and $a^{i},b^{ij}\in \tilde{C}^{\beta }(\mathbf{R}^{d}),G^{ij}\in \tilde{C}%
^{\beta \vee 1}(\mathbf{R}^{d})$. Then there exists a constant $C$ such that
for all $f\in \tilde{C}^{\beta -\alpha }(\mathbf{R}^{d}),$%
\begin{equation*}
\big\vert\mathbf{E}\big[f(Y_{s})-f(Y_{\tau _{i_{s}}})|\mathcal{F}_{\tau
_{i_{s}}}\big]\big\vert\leq C|f|_{\beta -\alpha }\delta ^{\frac{\beta }{%
\alpha }-1},\forall s\in \lbrack 0,T],
\end{equation*}%
where $i_{s}=i$ if $\tau _{i}\leq s<\tau _{i+1}$.
\end{lemma}

\begin{proof}
Applying It\^{o}'s formula, for $s\in \lbrack 0,T]$, 
\begin{equation*}
\mathbf{E}[f^{\varepsilon }(Y_{s})-f^{\varepsilon }(Y_{\tau _{i_{s}}})|%
\mathcal{F}_{\tau _{i_{s}}}]=\mathbf{E}\big[\int_{\tau _{i_{s}}}^{s}\big(%
L_{Y_{\tau _{i_{s}}}}f^{\varepsilon }(Y_{r})\big)dr\big|\mathcal{F}_{\tau
_{i_{s}}}\big].
\end{equation*}%
Hence, for $\varepsilon \in (0,1)$, by (\ref{ma8}) and (\ref{ma9}), 
\begin{eqnarray*}
|\mathbf{E}[f(Y_{s})-f(Y_{\tau _{i_{s}}})|\mathcal{F}_{\tau _{i_{s}}}]|
&\leq &|\mathbf{E}[(f-f^{\varepsilon })(Y_{s})-(f-f^{\varepsilon })(Y_{\tau
_{i_{s}}})|\mathcal{F}_{\tau _{i_{s}}}]| \\
&&+|\mathbf{E}[f^{\varepsilon }(Y_{s})-f^{\varepsilon }(Y_{\tau _{i_{s}}})|%
\mathcal{F}_{\tau _{i_{s}}}]| \\
&\leq &CF(\varepsilon ,\delta )|f|_{\beta -\alpha },
\end{eqnarray*}%
with a constant $C$ independent of $\varepsilon ,f$ and $F(\varepsilon
,\delta )=\varepsilon ^{\beta -\alpha }+\varepsilon ^{\beta -2\alpha }\delta
.$ Minimizing $F(\varepsilon ,\delta )$ in $\varepsilon \in (0,1)$, we
obtain 
\begin{equation*}
|\mathbf{E}[f(Y_{s})-f(Y_{\tau _{i_{s}}})|\mathcal{F}_{\tau _{i_{s}}}]|\leq
C\delta ^{\frac{\beta }{\alpha }-1}|f|_{\beta }.
\end{equation*}
\end{proof}

\subsection{Proof of Theorem \protect\ref{thm:main}}

Let $u\in \tilde{C}^{\beta }(H)$ be the unique solution to (\ref{eq1}) with $%
f=0$. By It\^{o}'s formula,%
\begin{eqnarray*}
\mathbf{E}[u(0,X_{0})] &=&\mathbf{E}[u(T,X_{T})]-\mathbf{E}\big[\int_{0}^{T}%
\big(\partial _{t}u(s,X_{s})+L_{X_{s}}u(s,X_{s})\big)ds\big] \\
&=&\mathbf{E}\big[g(X_{T})\big]
\end{eqnarray*}%
and 
\begin{equation}
\mathbf{E}[u(0,X_{0})]=\mathbf{E}[u(0,Y_{0})].  \label{eqn:expect_terminal}
\end{equation}

By Lemma \ref{le3}, 
\begin{equation}
|L_{z}u(s,\cdot )|_{\beta -\alpha }\leq C|g|_{\beta },|\partial
_{t}u(s,\cdot )|_{\beta -\alpha }\leq C|g|_{\beta },s\in \lbrack 0,T],z\in 
\mathbf{R}^{d}.  \label{maf90}
\end{equation}

Then, by It\^{o}'s formula and (\ref{maf90}), it follows that 
\begin{eqnarray*}
&&\mathbf{E}[g(Y_{T})]-\mathbf{E}[g(X_{T})]=\mathbf{E}[u(T,Y_{T})]-\mathbf{E}%
[u(0,Y_{0})] \\
&=&\mathbf{E}\Big[\int_{0}^{T}\Big\{\big[\partial _{t}u(s,Y_{s})-\partial
_{t}u(s,Y_{\tau _{i_{s}}})\big] \\
&&+\big[L_{Y_{\tau _{i_{s}}}}u(s,Y_{s})-L_{Y_{\tau _{i_{s}}}}u(s,Y_{\tau
_{i_{s}}})\big]\Big\}ds\Big].
\end{eqnarray*}

Hence, by (\ref{maf90}) and Lemma~\ref{lem:expect}, there exists a constant $%
C$ independent of $g$ such that 
\begin{equation*}
|\mathbf{E}g(Y_{T})-\mathbf{E}g(X_{T})|\leq C\delta ^{\frac{\beta }{\alpha }%
-1}|g|_{\beta }.
\end{equation*}%
The statement of Theorem \ref{thm:main} follows.

\subsubsection{Proof of Corollary \protect\ref{co2}}

According to \cite{re}, there is a rapidly decreasing smooth function $w\in 
\mathcal{S}(\mathbf{R}^{d})$, the Schwartz space, such that $\int w(x)dx=1$
and all moments are zero:%
\begin{equation*}
\int w(x)x^{\gamma }dx=0,\gamma \in \mathbf{N}^{d},\gamma \neq \mathbf{0},
\end{equation*}%
where $x^{\gamma }=x_{1}^{\gamma _{1}}\ldots x_{d}^{\gamma
_{d}},x=(x_{1},\ldots ,x_{d})\in \mathbf{R}^{d}$. Let $\varepsilon \in
(0,1),w_{\varepsilon }(x)=\varepsilon ^{-d}w(x/\varepsilon ),x\in \mathbf{R}%
^{d},$ 
\begin{equation*}
g_{\varepsilon }(x)=\int g(x-y)w_{\varepsilon }(y)dy,x\in \mathbf{R}^{d}.
\end{equation*}%
We will show that for $\beta \in (0,4],\nu \leq \beta ,$%
\begin{eqnarray}
\sup_{x}|g_{\varepsilon }(x)-g(x)| &\leq &C|g|_{\nu }\varepsilon ^{\nu },
\label{fo6} \\
|g_{\varepsilon }|_{\beta } &\leq &C\varepsilon ^{\nu -\beta }|g|_{\nu } 
\notag
\end{eqnarray}%
(A standard mollifier could be taken if $\nu \leq 2$, see Lemma \ref{lemn1}%
). Since for $x\in \mathbf{R}^{d},$%
\begin{equation*}
g_{\varepsilon }(x)-g(x)=\int [g(x-y)-g(x)-\sum_{1\leq |\gamma |\leq \lbrack
\nu ]^{-}}\frac{D^{\gamma }g(x)}{\gamma !}y^{\gamma }]w_{\varepsilon }(y)dy,
\end{equation*}%
it follows that 
\begin{equation*}
\sup_{x}|g_{\varepsilon }(x)-g(x)|\leq C|g|_{\nu }\varepsilon ^{\nu }.
\end{equation*}

If $\beta $ is an integer, $\gamma \in \mathbf{N}^{d},|\gamma |=\beta $ and $%
\gamma =\mu +\mu ^{\prime }$ with $|\mu |=[\nu ],\mu ^{\prime }\neq \mathbf{0%
}$, then 
\begin{eqnarray*}
D^{\gamma }g_{\varepsilon }(x) &=&\varepsilon ^{-[\gamma ]}\int
g(y)(D^{\gamma }w)_{\varepsilon }(x-y)dy=\varepsilon ^{\lbrack \nu ]-\beta
}\int D^{\mu }g(y)(D^{\mu ^{\prime }}w)_{\varepsilon }(x-y)dy \\
&=&\varepsilon ^{\lbrack \nu ]-\beta }\int [D^{\mu }g(y)-D^{\mu
}g(x)](D^{\mu ^{\prime }}w)_{\varepsilon }(x-y)dydy
\end{eqnarray*}%
and 
\begin{equation*}
|D^{\gamma }g_{\varepsilon }(x)|\leq C\varepsilon ^{\nu -\beta }|g|_{\nu
},x\in \mathbf{R}^{d}.
\end{equation*}%
If $\beta $ is not an integer, the second inequality in (\ref{fo6}) follows
by interpolation.

According to Theorem \ref{thm:main} and (\ref{fo6}), 
\begin{eqnarray*}
|\mathbf{E}g(Y_{T})-\mathbf{E}g(X_{T})| &\leq &2\sup_{x}|g_{\varepsilon
}(x)-g(x)|+|\mathbf{E}g_{\varepsilon }(Y_{T})-\mathbf{E}g_{\varepsilon
}(X_{T})| \\
&\leq &C|g|_{\nu }F(\varepsilon ,\delta ),
\end{eqnarray*}%
where $F(\varepsilon ,\delta )=\varepsilon ^{\nu }+\varepsilon ^{\nu -\beta
}\delta ^{\frac{\beta }{\alpha }-1}.$ Minimizing $F$ in $\varepsilon \in
(0,1)$, the statement of Corollary \ref{co2} follows.

\subsection{Approximate simple Euler scheme}

Consider the approximation of $X_{t}$ defined by the increments of $\tilde{Z}%
_{t}=Z_{t}^{\sigma }+R_{t}^{\sigma },0\leq t\leq T,$ in Example \ref{ex1}.
Obviously, $\tilde{Z}_{t}$ depends on $\alpha ,\beta $ and $\sigma $. Its
generator is 
\begin{eqnarray*}
\tilde{L}v(x) &=&\int_{0}^{t}\int_{|\upsilon |>\varepsilon }[v(s,x+\upsilon
)-v(s,x)-\chi _{\alpha }(\upsilon )\left( \nabla v(s,x),\upsilon \right)
]\pi (d\upsilon ) \\
&&+R^{\alpha ,\beta }v(x),
\end{eqnarray*}%
where%
\begin{equation*}
R^{\alpha ,\beta }v(x)=\left\{ 
\begin{array}{ll}
\int_{|\upsilon |\leq \sigma }\left( \nabla v(x),\upsilon \right) d\pi & 
\text{if }\alpha <\beta \in (1,2],\alpha \in (0,1], \\ 
&  \\ 
\frac{1}{2}\sum_{i,j}(B^{\sigma }{}^{\ast }B^{\sigma })_{ij}\partial
_{ij}^{2}v(x) & \text{if }\alpha <\beta \in (2,4],\alpha \in (1,2], \\ 
&  \\ 
0 & \text{otherwise.}%
\end{array}%
\right.
\end{equation*}

\begin{lemma}
\label{lemn2}Let $\alpha <\beta \leq 2\alpha $ and $h\in \tilde{C}^{\beta }(%
\mathbf{R}^{d})$. Then there is a constant $C$ such that for every $\mathbb{F%
}^{Z^{\sigma }}$-stopping times $0\leq \tau \leq \tau ^{\prime }\leq T$ $\ $%
we have%
\begin{equation*}
|\mathbf{E[}h(Z_{\tau ^{\prime }}-Z_{\tau })-h(\tilde{Z}_{\tau ^{\prime }}-%
\tilde{Z}_{\tau })|\mathcal{F}_{\tau }]|\leq C\phi (\sigma )|h|_{\beta }%
\mathbf{E}[\tau ^{\prime }-\tau |\mathcal{F}_{\tau }]\text{,}
\end{equation*}%
with%
\begin{equation*}
\phi (\sigma )=\int_{|\upsilon |\leq \sigma }|\upsilon |^{\beta \wedge 3}d\pi
\end{equation*}%
(here $\mathbb{F}^{Z^{\sigma }}$ is the natural filtration of $\sigma $%
-algebras generated by $Z^{\sigma }$).
\end{lemma}

\begin{proof}
Let $\bar{Z}^{\sigma }=Z-Z^{\sigma }$. We show first that there is a
constant $C$ such that for any $s<t,g\in \tilde{C}^{\beta }(\mathbf{R}^{d})$,%
\begin{equation}
|\mathbf{E}g(\bar{Z}_{t}^{\sigma }-\bar{Z}_{s}^{\sigma })-\mathbf{E}%
g(R_{t}^{\sigma }-R_{s}^{\sigma })|\leq C\phi (\sigma )|g|_{\beta }|t-s|.
\label{f1}
\end{equation}%
By Ito formula 
\begin{equation}
v(r,x)=\mathbf{E[}g(\bar{Z}_{t}^{\sigma }-\bar{Z}_{r}^{\sigma }+x),0\leq
r\leq t,  \label{fo1}
\end{equation}%
is the solution of the backward Kolmogorov equation%
\begin{eqnarray}
&&\partial _{t}v(r,x)+\int_{|\upsilon |\leq \sigma }[v(r,x+\upsilon
)-v(r,x)-\chi _{\alpha }(\upsilon )\left( \nabla v(r,x),\upsilon \right)
]\pi (d\upsilon )  \label{fo3} \\
&=&0,v(t,x)=g(x),0\leq s\leq t.  \notag
\end{eqnarray}%
Obviously, $v\in \tilde{C}^{\beta }([0,t]\times \mathbf{R}^{d})$ and (see (%
\ref{fo1})) $|v|_{\beta }\leq |g|_{\beta }.$ By Ito formula and (\ref{fo3}),%
\begin{eqnarray}
&&\mathbf{E}g(R_{t}^{\sigma }-R_{s}^{\sigma })-\mathbf{E}g(\bar{Z}%
_{t}^{\sigma }-\bar{Z}_{s}^{\sigma })  \label{f3} \\
&=&\mathbf{E}v(t,R_{t}^{\sigma }-R_{s}^{\sigma })-v(s,0)=\mathbf{E}%
\int_{s}^{t}[R^{\alpha ,\beta }v(r,R_{r}^{\sigma }-R_{s}^{\sigma })-\bar{L}%
v(r,R_{r}^{\sigma }-R_{s}^{\sigma })]dr,  \notag
\end{eqnarray}%
where%
\begin{equation*}
\bar{L}v(r,x)=\int_{|\upsilon |\leq \sigma }[v(r,x+\upsilon )-v(r,x)-\chi
_{\alpha }(\upsilon )\left( \nabla v(r,x),\upsilon \right) ]\pi (d\upsilon
),(r,x)\in H.
\end{equation*}

If $\alpha <\beta \in (1,2],\alpha \in (0,1],$ then for all $(r,x)\in H,$ 
\begin{eqnarray*}
&&|R^{\alpha ,\beta }v(r,x)-\int_{|\upsilon |\leq \sigma }[v(r,x+\upsilon
)-v(r,x)]\pi (d\upsilon )| \\
&\leq &\int_{0}^{1}\int_{|\upsilon |\leq \sigma }|\nabla v(r,x+s\upsilon
)-\nabla v(r,x)|~|\upsilon |d\pi ds \\
&\leq &C|v|_{\beta }\int_{|\upsilon |\leq \sigma }|\upsilon |^{\beta }d\pi
\leq C|h|_{\beta }\int_{|\upsilon |\leq \sigma }|\upsilon |^{\beta }d\pi .
\end{eqnarray*}%
\newline
If $\alpha <\beta \in (2,4],\alpha \in (1,2],$ then for all $(r,x)\in H,$ 
\begin{eqnarray*}
&&|R^{\alpha ,\beta }v(r,x)-\int_{|\upsilon |\leq \sigma }[v(r,x+\upsilon
)-v(r,x)-(\nabla v(r,x),\upsilon )]d\pi | \\
&\leq &\int_{0}^{1}\int_{|\upsilon |\leq \sigma }|D^{2}v(r,x+s\upsilon
)-D^{2}v(r,x)|~|\upsilon |^{2}d\pi ds \\
&\leq &C|v|_{\beta }\int_{|\upsilon |\leq \sigma }|\upsilon |^{\beta \wedge
3}d\pi \leq C|h|_{\beta }\int_{|\upsilon |\leq \sigma }|\upsilon |^{\beta
\wedge 3}d\pi .
\end{eqnarray*}%
The estimate of the difference $R^{\alpha ,\beta }v-\bar{L}v$ in the other
cases is straightforward and (\ref{f1}) follows by (\ref{f3}).

Since $Z^{\sigma },\bar{Z}^{\sigma }$ and $R^{\sigma }$ are independent and $%
\tau ,\tau ^{\prime }$ are $\mathbb{F}^{Z^{\sigma }}$ stopping times, we
have by (\ref{f1}) that%
\begin{eqnarray*}
&&|\mathbf{E[}h(Z_{\tau ^{\prime }}^{\sigma }-Z_{\tau }^{\sigma }+\bar{Z}%
_{\tau ^{\prime }}^{\sigma }-\bar{Z}_{\tau }^{\sigma })-h(Z_{\tau ^{\prime
}}^{\sigma }-Z_{\tau }^{\sigma }+R_{\tau ^{\prime }}^{\sigma }-R_{\tau
}^{\sigma })|\mathcal{F}_{\tau }]| \\
&\leq &C\phi (\sigma )|h|_{\beta }\mathbf{E}[\tau ^{\prime }-\tau |\mathcal{F%
}_{\tau }]\text{.}
\end{eqnarray*}%
The statement follows.
\end{proof}

For the proof or Theorem \ref{main2} we will need the following estimate.

\begin{lemma}
\label{lemn3}Let%
\begin{equation*}
V_{t}=at+bW_{t}+GZ_{t},
\end{equation*}%
where $a\in \mathbf{R}^{d}$, $b$ is a $d\times d$-matrix and $G$ is a $%
m\times m$-matrix. We assume $b=0$ if $\alpha \in (0,2)$ and $a=0$ if $%
\alpha \in (0,1)$ and 
\begin{equation*}
|a|+|b|+|G|\leq K.
\end{equation*}%
Let $\alpha <\beta \leq \mu \leq 2\alpha $ and $h\in \tilde{C}^{\beta
-\alpha }(\mathbf{R}^{d})$.

Then there is a constant $C=C(\alpha ,\beta ,K)$ such that%
\begin{equation*}
|\mathbf{E}h(V_{t})-h(0)|\leq Ct^{\frac{\beta }{\alpha }-1}|h|_{\beta
-\alpha }\text{. }
\end{equation*}
\end{lemma}

\begin{proof}
For $f\in \tilde{C}^{\beta }(\mathbf{R}^{d})$, applying Ito formula,%
\begin{equation*}
\mathbf{E}f(V_{t})-f(0)=\mathbf{E}\int_{0}^{t}\mathcal{K}f(V_{r})dr\text{,}
\end{equation*}%
where for $x\in \mathbf{R}^{d},$%
\begin{eqnarray*}
\mathcal{K}f(x) &=&(a,\nabla f(x))+\frac{1}{2}\sum_{i,j}b^{\ast }b\partial
_{ij}^{2}f(x) \\
&&+\int [f(x+\upsilon )-f(x)-\chi _{\alpha }(\upsilon )(\nabla f(x),\upsilon
)]\pi (d\upsilon ).
\end{eqnarray*}%
For $h\in \tilde{C}^{\beta -\alpha }(\mathbf{R}^{d})$ we take $w\in
C_{0}^{\infty }(\mathbf{R}^{d}),$ be a nonnegative smooth function with
support in $\{|x|\leq 1\}$ such that $w(x)=w(|x|)$, $x\in \mathbf{R}^{d},$
and $\int w(x)dx=1.$ For $x\in \mathbf{R}^{d}$ and $\varepsilon \in (0,1)$,
define $w^{\varepsilon }(x)=\varepsilon ^{-d}w\left( \frac{x}{\varepsilon }%
\right) $ and the convolution 
\begin{equation*}
h^{\varepsilon }(x)=\int f(y)w^{\varepsilon }(x-y)dy,x\in \mathbf{R}^{d}.
\end{equation*}%
Then by Lemma \ref{lemn1} 
\begin{eqnarray*}
|\mathbf{E}h(V_{t})-h(0)| &\leq &2\varepsilon ^{\beta -\alpha }|h|_{\beta
-\alpha }+|\mathbf{E}\int_{0}^{t}\mathcal{K}h^{\varepsilon }(V_{r})dr| \\
&\leq &C|h|_{\beta -\alpha }(\varepsilon ^{\beta -\alpha }+\varepsilon
^{\beta -2\alpha }t)
\end{eqnarray*}%
for each $\varepsilon \in (0,1).$ The statement follows by minimizing the
inequality in $\varepsilon .$
\end{proof}

\subsubsection{Proof of Theorem \protect\ref{main2}}

Let $u\in \tilde{C}^{\beta }(H)$ be the unique solution to the backward
Kolmogorov equation%
\begin{eqnarray}
\big(\partial _{t}+L\big)u(t,x) &=&0,  \label{fo44} \\
u(T,x) &=&g(x).  \notag
\end{eqnarray}%
Let for $\tau _{i}\leq t\leq \tau _{i+1}$ 
\begin{eqnarray*}
H_{t}^{i} &=&a(\tilde{Y}_{\tau _{i}})(t-\tau _{i})+b(\tilde{Y}_{\tau
_{i}})(W_{t}-W_{\tau _{i}}) \\
&&+G(\tilde{Y}_{\tau _{i}})\left( Z_{t}-Z_{\tau _{i}}\right)
\end{eqnarray*}%
and denote $\Delta \tilde{Y}_{\tau _{i}}=\tilde{Y}_{\tau _{i+1}}-\tilde{Y}%
_{\tau _{i}}$. We approximate%
\begin{eqnarray*}
&&u(T,\tilde{Y}_{T})-u(0,Y_{0}) \\
&=&\sum_{i}u(\tau _{i+1},\tilde{Y}_{\tau _{i+1}})-u(\tau _{i},\tilde{Y}%
_{\tau _{i}}) \\
&=&\sum_{i}[u(\tau _{i+1},\tilde{Y}_{\tau _{i}}+\Delta \tilde{Y}_{\tau
_{i}})-u(\tau _{i+1},\tilde{Y}_{\tau _{i}}+H_{\tau _{i+1}}^{i})] \\
&&+\sum_{i}[u(\tau _{i+1},\tilde{Y}_{\tau _{i}}+H_{\tau _{i+1}}^{i})-u(\tau
_{i},\tilde{Y}_{\tau _{i}})] \\
&=&D_{1}+\sum_{i}D_{2i}.
\end{eqnarray*}%
According to (\ref{fo5}) (Lemma \ref{lemn2}),%
\begin{equation*}
\mathbf{E}|D_{1}|\leq C\phi (\sigma )|u|_{\beta }\leq C\phi (\sigma
)|g|_{\beta }.
\end{equation*}%
Now, we estimate the second term. By Ito formula for each $i,$%
\begin{eqnarray*}
\mathbf{E[}D_{2i}|\mathcal{F}_{\tau _{i}}] &=&\mathbf{E}[u(\tau _{i+1},%
\tilde{Y}_{\tau _{i}}+H_{\tau _{i+1}}^{i})-u(\tau _{i+1},\tilde{Y}_{\tau
_{i}})|\mathcal{F}_{\tau _{i}}] \\
&=&\mathbf{E\{}\int_{\tau _{i}}^{\tau _{i+1}}[\partial _{t}u(r,\tilde{Y}%
_{\tau _{i}}+H_{r}^{i})+L_{\tilde{Y}_{\tau _{i}}}u(r,\tilde{Y}_{\tau
_{i}}+H_{r}^{i})]dr|\mathcal{F}_{\tau _{i}}\} \\
&=&\mathbf{E}\int_{\tau _{i}}^{\tau _{i+1}}[(\partial _{t}u(r,\tilde{Y}%
_{\tau _{i}}+H_{r}^{i})-\partial _{t}u(r,\tilde{Y}_{\tau _{i}})) \\
&&+(L_{\tilde{Y}_{\tau _{i}}}u(r,\tilde{Y}_{\tau _{i}}+H_{r}^{i})-L_{\tilde{Y%
}_{\tau _{i}}}u(r,\tilde{Y}_{\tau _{i}}))]dr
\end{eqnarray*}%
and by\ Theorem \ref{thm:StoCP} and Lemmas \ref{le3} and \ref{lemn3},%
\begin{eqnarray*}
\left\vert \sum_{i}\mathbf{E}D_{2i}\right\vert &\leq &\sum_{i}|\mathbf{E}%
D_{2i}|\leq C\delta ^{\frac{\beta }{\alpha }-1}|Lu|_{\beta -\alpha } \\
&\leq &C\delta ^{\frac{\beta }{\alpha }-1}|u|_{\beta }\leq C\delta ^{\frac{%
\beta }{\alpha }-1}|g|_{\beta }
\end{eqnarray*}%
and the statement of Theorem \ref{main2} follows.

\subsection{Approximate jump-adapted scheme}

Consider the approximation of $X_{t}$ defined by the increments of $\tilde{Z}%
_{t}=Z_{t}^{\sigma }+R_{t}^{\sigma },0\leq t\leq T,$ in Example \ref{ex1}.
For $\sigma \in (0,1),\delta >0,$ consider the following $Z^{\sigma }$-jump
adapted time discretization: $\tau _{0}=0,$%
\begin{equation*}
\tau _{i+1}=\inf \left( t>\tau _{i}:\Delta Z_{t}^{\sigma }\neq 0\right)
\wedge (\tau _{i}+\delta )\wedge T.
\end{equation*}%
In this case the time discretization $\{\tau _{i},i=0,\ldots ,n_{T}\}$ of
the interval $[0,T]$ is random, $\tau _{i}$ are stopping times. We
approximate $X_{t}$ by%
\begin{equation*}
\hat{Y}_{t}=X_{0}+\int_{0}^{t}a(\hat{Y}_{\tau _{i_{s}}})ds+\int_{0}^{t}b(%
\hat{Y}_{\tau _{i_{s}}})dW_{s}+\int_{0}^{t}G(\hat{Y}_{\tau _{i_{s}}})d\tilde{%
Z}_{s},t\in \lbrack 0,T].
\end{equation*}

In this case,%
\begin{equation*}
\tau _{i+1}-\tau _{i}=\eta _{i+1}\wedge \delta \wedge (T-\tau _{i})
\end{equation*}%
with%
\begin{equation*}
\eta _{i+1}=\inf (t>0:p\left( (\tau _{i},\tau _{i}+t],\left\{ |\upsilon
|>\sigma \right\} \right) \geq 1)
\end{equation*}%
and $\eta _{i+1}$ is $\mathcal{F}_{\tau _{i}}$-conditionally exponential
with parameter $\lambda _{\sigma }=\pi \left( \left\{ |\upsilon |>\sigma
\right\} \right) $.

\begin{lemma}
\label{le2}Let $\delta _{i}^{\prime }=\delta \wedge (T-\tau _{i}),i\geq 0,$
and $\lambda _{\sigma }=\pi \left( \left\{ |\upsilon |>\sigma \right\}
\right) $.

(i) There is constant $c>0$ such that for any $i\geq 0$%
\begin{equation*}
c\left( \delta _{i}^{\prime }\wedge \lambda _{\sigma }^{-1}\right) \leq 
\mathbf{E}[\tau _{i+1}-\tau _{i}|\mathcal{F}_{\tau _{i}}]\leq \delta
_{i}^{\prime }\wedge \lambda _{\sigma }^{-1}.
\end{equation*}

(ii) There is a constant $C$ such that for any $i\geq 0,$%
\begin{eqnarray*}
\mathbf{E}[(\tau _{i+1}-\tau _{i})^{2}|\mathcal{F}_{\tau _{i}}] &\leq &C%
\mathbf{E}[\delta _{i}^{\prime 2}\wedge \lambda _{\sigma }^{-2}|\mathcal{F}%
_{\tau _{i}}] \\
&\leq &C(\delta \wedge \lambda _{\sigma }^{-1})\mathbf{E}[\tau _{i+1}-\tau
_{i}|\mathcal{F}_{\tau _{i}}].
\end{eqnarray*}
\end{lemma}

\begin{proof}
Since $\tau _{i+1}-\tau _{i}=\eta _{i+1}\wedge \delta \wedge (T-\tau _{i})$
and%
\begin{equation*}
\eta _{i+1}=\inf (t>0:p\left( (\tau _{i},\tau _{i}+t],\left\{ |\upsilon
|>\sigma \right\} \right) \geq 1)
\end{equation*}%
is $\mathcal{F}_{\tau _{i}}$-conditionally exponential with parameter $%
\lambda _{\sigma },$ we find%
\begin{eqnarray*}
\mathbf{E}[\tau _{i+1}-\tau _{i}|\mathcal{F}_{\tau _{i}}] &=&\mathbf{E}\left[
\eta _{i+1}\wedge \delta _{i}^{\prime }|\mathcal{F}_{\tau _{i}}\right]
=\lambda _{\sigma }\int_{0}^{\delta _{i}^{\prime }}te^{-\lambda _{\sigma
}t}dt+\delta _{i}^{\prime }e^{-\lambda _{\sigma }\delta _{i}^{\prime }} \\
&=&\frac{1-e^{-\lambda _{\sigma }\delta _{i}^{\prime }}}{\lambda _{\sigma }}.
\end{eqnarray*}%
If $\delta _{i}^{\prime }\geq \lambda _{\sigma }^{-1}$, then $\delta
_{i}^{\prime }\lambda _{\sigma }\geq 1$ and%
\begin{equation*}
\frac{1-e^{-\lambda _{\sigma }\delta _{i}^{\prime }}}{\lambda _{\sigma }}%
\geq \frac{1-e^{-1}}{\lambda _{\sigma }}\geq \frac{1}{3}\lambda _{\sigma
}^{-1}.
\end{equation*}%
If $\delta _{i}^{\prime }\leq \lambda _{\sigma }^{-1}$, then $\delta
_{i}^{\prime }\lambda _{\sigma }\leq 1$ and%
\begin{equation*}
\frac{1-e^{-\lambda _{\sigma }\delta _{i}^{\prime }}}{\lambda _{\sigma }}=%
\frac{1-e^{-\lambda _{\sigma }\delta _{i}^{\prime }}}{\lambda _{\sigma
}\delta _{i}^{\prime }}\delta _{i}^{\prime }\geq \frac{1}{2}\delta
_{i}^{\prime }.
\end{equation*}%
Therefore (i) follows. Similarly,%
\begin{eqnarray*}
\mathbf{E}[(\tau _{i+1}-\tau _{i})^{2}|\mathcal{F}_{\tau _{i}}] &=&\lambda
_{\sigma }\mathbf{E}[\int_{0}^{\delta _{i}^{\prime }}t^{2}e^{-\lambda
_{\sigma }t}dt+\delta _{i}^{\prime 2}e^{-\lambda _{\sigma }\delta
_{i}^{\prime }}|\mathcal{F}_{\tau _{i}}]dt \\
&=&\frac{2}{\lambda _{\sigma }^{2}}[-\lambda _{\sigma }\delta _{i}^{\prime
}e^{-\lambda _{\sigma }\delta _{i}^{\prime }}+1-e^{-\lambda _{\sigma }\delta
_{i}^{\prime }}]
\end{eqnarray*}%
and (ii) follows using (i).
\end{proof}

An immediate consequence of Lemma \ref{le2} is the following statement.

\begin{corollary}
\label{co5}(i) There are constants $c,C>0$ such that%
\begin{eqnarray*}
c\mathbf{E}\sum_{i}(\tau _{i+1}-\tau _{i}) &\leq &\sum_{i}\mathbf{E[}(\delta
\wedge \lambda _{\sigma }^{-1})\wedge (T-\tau _{i})] \\
&\leq &C\mathbf{E}\sum_{i}(\tau _{i+1}-\tau _{i})=CT.
\end{eqnarray*}

(ii) There is $C>0$ such that%
\begin{equation*}
\sum_{i}\mathbf{E[}(\tau _{i+1}-\tau _{i})^{2}]\leq CT(\delta \wedge \lambda
_{\sigma }^{-1}).
\end{equation*}
\end{corollary}

\begin{proof}
We derive (i) by summing inequalities in Lemma \ref{le2}(i)$.$ According to
Lemma \ref{le2}(ii) and (i), 
\begin{eqnarray*}
\sum_{i}\mathbf{E[}(\tau _{i+1}-\tau _{i})^{2}] &\leq &C\sum_{i}\mathbf{E}[(T%
\mathbf{-}\tau _{i})^{2}\wedge \delta ^{2}\wedge \lambda _{\sigma }^{-2}] \\
&\leq &C(T\wedge \delta \wedge \lambda _{\sigma }^{-1})\sum_{i}\mathbf{E}[(T%
\mathbf{-}\tau _{i})\wedge \delta \wedge \lambda _{\sigma }^{-1}] \\
&\leq &CT(\delta \wedge \lambda _{\sigma }^{-1}).
\end{eqnarray*}%
The statement follows.
\end{proof}

For the proof of Theorem \ref{propa2} we will need the following estimate as
well.

\begin{lemma}
\label{lemn30}Let%
\begin{equation*}
V_{t}=at+bW_{t}+GZ_{t},
\end{equation*}%
where $a\in \mathbf{R}^{d}$, $b$ is a $d\times d$-matrix and $G$ is a $%
m\times m$-matrix. We assume $b=0$ if $\alpha \in (0,2)$ and $a=0$ if $%
\alpha \in (0,1)$ and 
\begin{equation*}
|a|+|b|+|G|\leq K.
\end{equation*}%
Let $\alpha <\beta \leq \mu \leq 2\alpha $ and $h\in \tilde{C}^{\beta
-\alpha }(\mathbf{R}^{d})$.

Then there is a constant $C=C(\alpha ,\beta ,K)$ such that for any $i\geq 0$%
\begin{equation*}
|\mathbf{E[}\int_{\tau _{i}}^{\tau _{i+1}}h(V_{r})-h(V_{\tau _{i}})|\mathcal{%
F}_{\tau _{i}}]|\leq C|h|_{\beta -\alpha }\tilde{\lambda}_{\sigma }^{\frac{%
\beta }{\alpha }-1}\left( \delta \wedge \lambda _{\sigma }^{-1}\right) ^{%
\frac{\beta }{\alpha }-1}\mathbf{E}[(\tau _{i+1}-\tau _{i})|\mathcal{F}%
_{\tau _{i}}]\text{, }
\end{equation*}%
where $\lambda _{\sigma }=\pi \left( \left\{ |\upsilon |>\sigma \right\}
\right) ,$ 
\begin{equation*}
\tilde{\lambda}_{\sigma }=1+1_{\alpha \in (1,2)}|\int_{1\geq |\upsilon
|>\sigma }\upsilon d\pi |.
\end{equation*}
\end{lemma}

\begin{proof}
For $f\in \tilde{C}^{\beta }(\mathbf{R}^{d}),i\geq 0,$ applying Ito formula,%
\begin{equation*}
\mathbf{E[}\int_{\tau _{i}}^{\tau _{i+1}}f(V_{r})-f(V_{\tau _{i}})|\mathcal{F%
}_{\tau _{i}}]dr=\mathbf{E}\int_{\tau _{i}}^{\tau _{i+1}}[\int_{\tau
_{i}}^{s}\mathcal{K}f(V_{r})dr+M_{s}-M_{\tau _{i}}]ds|\mathcal{F}_{\tau
_{i}}]dr,
\end{equation*}%
where for $x\in \mathbf{R}^{d},$%
\begin{eqnarray*}
\mathcal{K}f(x) &=&(a,\nabla f(x))+\frac{1}{2}\sum_{i,j}b^{\ast }b\partial
_{ij}^{2}f(x) \\
&&+\int [f(x+\upsilon )-f(x)-\chi _{\alpha }(\upsilon )(\nabla f(x),\upsilon
)]\pi (d\upsilon )
\end{eqnarray*}%
and%
\begin{equation*}
M_{t}=\int_{0}^{t}\int [f(V_{r-}+G\upsilon )-f(V_{r-})]q(dr,d\upsilon ),t\in
\lbrack 0,T]
\end{equation*}%
Note that%
\begin{equation*}
\int_{\tau _{i}}^{\tau _{i+1}}(M_{s}-M_{\tau _{i}})d(s-\tau _{i})=(M_{\tau
_{i+1}}-M_{\tau _{i}})(\tau _{i+1}-\tau _{i})-\int_{\tau _{i}}^{\tau
_{i+1}}(s-\tau _{i})dM_{s}.
\end{equation*}%
Since $Z^{\sigma }$and $\bar{Z}^{\sigma }=Z-Z^{\sigma }$ are independent and 
$\tau _{i}$ are $\mathbb{F}^{Z^{\sigma }}$-stopping times, it follows by
definition of $\tau _{i}$ that%
\begin{eqnarray*}
&&\mathbf{E[}(M_{\tau _{i+1}}-M_{\tau _{i}})(\tau _{i+1}-\tau
_{i})-\int_{\tau _{i}}^{\tau _{i+1}}(s-\tau _{i})dM_{s}|\mathcal{F}_{\tau
_{i}}] \\
&=&\mathbf{E}[-(\tau _{i+1}-\tau _{i})(U_{\tau _{i+1}}^{\sigma }-U_{\tau
_{i}}^{\sigma })+\int_{\tau _{i}}^{\tau _{i+1}}(s-\tau _{i})dU_{s}^{\sigma }|%
\mathcal{F}_{\tau _{i}}] \\
&=&-\mathbf{E[}\int_{\tau _{i}}^{\tau _{i+1}}(U_{s}^{\sigma }-U_{\tau
_{i}}^{\sigma })ds|\mathcal{F}_{\tau _{i}}],
\end{eqnarray*}%
where%
\begin{eqnarray*}
U_{t}^{\sigma } &=&\int_{0}^{t}\int_{|\upsilon |>\sigma }[f(V_{r-}+G\upsilon
)-f(V_{r-})]d\pi dr \\
&=&\int_{0}^{t}\int_{|\upsilon |>1}[f(V_{r-}+G\upsilon )-f(V_{r-})]d\pi
dr+\int_{0}^{t}\int_{1\geq |\upsilon |>\sigma }\chi _{\alpha }(\upsilon
)(\nabla f(V_{r}),\upsilon )d\pi dr \\
&&+\int_{0}^{t}\int_{1\geq |\upsilon |>\sigma }[f(V_{r-}+G\upsilon
)-f(V_{r-})-\chi _{\alpha }(\upsilon )(\nabla f(V_{r}),\upsilon )]d\pi dr
\end{eqnarray*}%
Hence%
\begin{eqnarray}
&&|\mathbf{E[}\int_{\tau _{i}}^{\tau _{i+1}}f(V_{r})-f(V_{\tau _{i}})|%
\mathcal{F}_{\tau _{i}}]dr|  \label{for7} \\
&\leq &C(1+1_{\alpha \in (1,2)}|\int_{1\geq |\upsilon |>\varepsilon
}\upsilon d\pi |)|f|_{\beta }\mathbf{E}[(\tau _{i+1}-\tau _{i})^{2}|\mathcal{%
F}_{\tau _{i}}].  \notag
\end{eqnarray}

For $h\in \tilde{C}^{\beta -\alpha }(\mathbf{R}^{d})$ we take $w\in
C_{0}^{\infty }(\mathbf{R}^{d}),$ be a nonnegative smooth function with
support in $\{|x|\leq 1\}$ such that $w(x)=w(|x|)$, $x\in \mathbf{R}^{d},$
and $\int w(x)dx=1.$ For $x\in \mathbf{R}^{d}$ and $\varepsilon \in (0,1)$,
define $w^{\varepsilon }(x)=\varepsilon ^{-d}w\left( \frac{x}{\varepsilon }%
\right) $ and the convolution 
\begin{equation*}
h^{\varepsilon }(x)=\int f(y)w^{\varepsilon }(x-y)dy,x\in \mathbf{R}^{d}.
\end{equation*}%
Then by Lemma \ref{lemn1} and (\ref{for7}),%
\begin{eqnarray*}
&&\left\vert \mathbf{E[}\int_{\tau _{i}}^{\tau _{i+1}}h(V_{r})-h(V_{\tau })|%
\mathcal{F}_{\tau _{i}}]dr\right\vert \\
&\leq &2\varepsilon ^{\beta -\alpha }|h|_{\beta -\alpha }\mathbf{E[}\left(
\tau _{i+1}-\tau _{i}\right) |\mathcal{F}_{\tau }]+|\mathbf{E[}\int_{\tau
_{i}}^{\tau _{i+1}}(h^{\varepsilon }(V_{r})-h^{\varepsilon }(V_{\tau }))dr|%
\mathcal{F}_{\tau }]| \\
&\leq &2\varepsilon ^{\beta -\alpha }|h|_{\beta -\alpha }\mathbf{E[}\tau
_{i+1}-\tau _{i}|\mathcal{F}_{\tau _{i}}] \\
&&+C\varepsilon ^{\beta -2\alpha }[1+1_{\alpha \in (1,2)}|\int_{1\geq
|\upsilon |>\varepsilon }\upsilon d\pi |]|h|_{\beta -\alpha }\mathbf{E}%
[(\tau _{i+1}-\tau _{i})^{2}|\mathcal{F}_{\tau _{i}}].
\end{eqnarray*}%
Minimizing the inequality in $\varepsilon $ we find by Lemma \ref{le2}(ii)
that%
\begin{eqnarray*}
&&\left\vert \mathbf{E[}\int_{\tau _{i}}^{\tau _{i+1}}h(V_{r})-h(V_{\tau })|%
\mathcal{F}_{\tau _{i}}]dr\right\vert \\
&\leq &C|h|_{\beta -\alpha }\tilde{\lambda}_{\sigma }^{\frac{\beta }{%
^{\alpha }}-1}\mathbf{E[}\tau _{i+1}-\tau _{i}|\mathcal{F}_{\tau _{i}}]^{2-%
\frac{\beta }{\alpha }}\mathbf{E}[(\tau _{i+1}-\tau _{i})^{2}|\mathcal{F}%
_{\tau _{i}}]^{\frac{\beta }{\alpha }-1} \\
&\leq &C|h|_{\beta -\alpha }\tilde{\lambda}_{\sigma }^{\frac{\beta }{%
^{\alpha }}-1}(\delta \wedge \lambda _{\sigma }^{-1})^{\frac{\beta }{\alpha }%
-1}\mathbf{E[}\tau _{i+1}-\tau _{i}|\mathcal{F}_{\tau _{i}}].
\end{eqnarray*}
\end{proof}

\subsubsection{Proof of Theorem \protect\ref{propa2}}

Let $u\in \tilde{C}^{\beta }(H)$ be the unique solution to the backward
Kolmogorov equation (see Theorem \ref{thm:StoCP})%
\begin{eqnarray}
\big(\partial _{t}+L\big)u(t,x) &=&0,  \label{fo4} \\
u(T,x) &=&g(x).  \notag
\end{eqnarray}%
Let for $\tau _{i}\leq t\leq \tau _{i+1}$ 
\begin{eqnarray*}
H_{t}^{i} &=&a(\hat{Y}_{\tau _{i}})(t-\tau _{i})+b(\hat{Y}_{\tau
_{i}})(W_{t}-W_{\tau _{i}}) \\
&&+G(\hat{Y}_{\tau _{i}})\left( Z_{t}-Z_{\tau _{i}}\right)
\end{eqnarray*}%
and denote $\Delta \hat{Y}_{\tau _{i}}=\hat{Y}_{\tau _{i+1}}-\hat{Y}_{\tau
_{i}}$. We approximate%
\begin{eqnarray*}
&&u(T,\hat{Y}_{T})-u(0,X_{0}) \\
&=&\sum_{i}u(\tau _{i+1},\hat{Y}_{\tau _{i+1}})-u(\tau _{i},\hat{Y}_{\tau
_{i}}) \\
&=&\sum_{i}[u(\tau _{i+1},\hat{Y}_{\tau _{i}}+\Delta \hat{Y}_{\tau
_{i}})-u(\tau _{i+1},\hat{Y}_{\tau _{i}}+H_{\tau _{i+1}}^{i})] \\
&&+\sum_{i}[u(\tau _{i+1},\hat{Y}_{\tau _{i}}+H_{\tau _{i+1}}^{i})-u(\tau
_{i},\hat{Y}_{\tau _{i}})] \\
&=&D_{1}+\sum_{i}D_{2i}.
\end{eqnarray*}%
According to Lemma \ref{lemn2}, 
\begin{equation*}
\mathbf{E}|D_{1}|\leq C\phi (\sigma )|u|_{\beta }\leq C\phi (\sigma
)|g|_{\beta }.
\end{equation*}%
Now, we estimate the second term. By Ito formula for each $i,$%
\begin{eqnarray*}
\mathbf{E[}D_{2i}|\mathcal{F}_{\tau _{i}}] &=&\mathbf{E}[u(\tau _{i+1},\hat{Y%
}_{\tau _{i}}+H_{\tau _{i+1}}^{i})-u(\tau _{i+1},\hat{Y}_{\tau _{i}})|%
\mathcal{F}_{\tau _{i}}] \\
&=&\mathbf{E\{}\int_{\tau _{i}}^{\tau _{i+1}}[\partial _{t}u(r,\hat{Y}_{\tau
_{i}}+H_{r}^{i})+L_{\hat{Y}_{\tau _{i}}}u(r,\hat{Y}_{\tau
_{i}}+H_{r}^{i})]dr|\mathcal{F}_{\tau _{i}}\} \\
&=&\mathbf{E}\int_{\tau _{i}}^{\tau _{i+1}}[(\partial _{t}u(r,\hat{Y}_{\tau
_{i}}+H_{r}^{i})-\partial _{t}u(r,\hat{Y}_{\tau _{i}})) \\
&&+(L_{\hat{Y}_{\tau _{i}}}u(r,\hat{Y}_{\tau _{i}}+H_{r}^{i})-L_{\hat{Y}%
_{\tau _{i}}}u(r,\hat{Y}_{\tau _{i}}))]dr
\end{eqnarray*}%
and by\ Theorem \ref{thm:StoCP} and Lemmas \ref{le3}, \ref{lemn30} and
Corollary \ref{co5},%
\begin{eqnarray*}
\left\vert \sum_{i}\mathbf{E}D_{2i}\right\vert &\leq &\sum_{i}|\mathbf{E}%
D_{2i}|\leq C\tilde{\lambda}_{\sigma }^{\frac{\beta }{\alpha }-1}\left(
\delta \wedge \lambda _{\sigma }^{-1}\right) ^{\frac{\beta }{\alpha }%
-1}(|\partial _{t}u|_{\beta -\alpha }+|Lu|_{\beta -\alpha }) \\
&\leq &C\tilde{\lambda}_{\sigma }^{\frac{\beta }{\alpha }-1}\left( \delta
\wedge \lambda _{\sigma }^{-1}\right) ^{\frac{\beta }{\alpha }-1}|u|_{\beta
}\leq C\tilde{\lambda}_{\sigma }^{\frac{\beta }{\alpha }-1}\left( \delta
\wedge \lambda _{\sigma }^{-1}\right) ^{\frac{\beta }{\alpha }-1}|g|_{\beta }
\end{eqnarray*}%
and the statement of Theorem \ref{propa2} follows.

\section{Conclusion}

The paper studies a simple weak Euler approximation of solutions to possibly
completely degenerate stochastic differential equations driven by L\'{e}vy
processes. The dependence of the rate of convergence on the regularity of
coefficients and driving processes is investigated under the assumption of $%
\beta $-Lipshitz continuity of the coefficients. It is assumed that the SDE
is driven by Levy processes of order $\alpha \in (0,2]$ and that the tail of
the L\'{e}vy measure of the driving process has a $\mu $-order finite moment
($\mu \in (\alpha ,2\alpha ]).$ The resulting rate depends on $\beta ,\alpha 
$ and $\mu $. Following \cite{jkmp}, the robustness of the results to the
approximation of the law of the increments of the driving noise is studied
as well. It is shown that time discretization and substitution errors add
up. In addition, a jump-adapted approximate Euler scheme is considered as
well. The derived error estimate shows that sometimes the inclusion of jump
moments into time discretization $\left\{ \tau _{i}\right\} $ could improve
the convergence rate. In order to estimate the rate of convergence, the
existence of a unique solution to the corresponding backward degenerate
Kolmogorov equation in Lipshitz space is first proved.

On the other hand, there is a discrepancy in the model (\ref{one}) between $%
\alpha =2$ and $\alpha \in (0,2).$ One would like to consider the equation 
\begin{equation*}
X_{t}=X_{0}+\int_{0}^{t}a(X_{s})ds+\int_{0}^{t}b(X_{s})dW_{s}^{\alpha
}+\int_{0}^{t}G(X_{s-})dZ_{s},t\in \lbrack 0,T],
\end{equation*}%
with a possibly degenerate $b$ and a spherically symmetric $\alpha $-stable $%
W^{\alpha }$ (in (\ref{one}), $b=0$ for $\alpha \in (0,2)$).

Since (\ref{one}) could be degenerate, a solution corresponding to a given $%
\alpha \in (0,2]$ can be looked at as a solution corresponding to $\bar{%
\alpha}\in (\alpha ,2]$ as well. Therefore the rate for a fixed $\alpha $
cannot be "universally optimal" : there is always a large subclass for which
the rate claimed for $\alpha $ could be better and achieved under weaker
assumptions. For example, if $\beta =\mu =2\alpha $ with $\alpha \in (0,2)$
(the diffusion part is absent), the convergence order is $\kappa =1$ ($\mu
=4 $ and $G\in \tilde{C}^{4}$ is not needed). Even "strictly at $\alpha "$,
the assumption about the tail moment $\mu \in (\alpha ,2\alpha ]$ is not
optimal. It could be weakened for a subclass with the driving processes $Z$
such that the compensator of the jump measure of $X_{t}$ has a nice density
with respect to a reference measure. For example, let us consider the
following one dimensional model%
\begin{equation}
X_{t}=X_{0}+\int_{0}^{t}a(X_{s})ds+\int_{0}^{t}b(X_{s})dW_{s}+%
\int_{0}^{t}G(X_{s-})dZ_{s},t\in \lbrack 0,T],  \label{fo7}
\end{equation}%
where $Z$ is a symmetric $\lambda $-stable with $\lambda \in (0,1)$ and $%
G\geq 0.$ Assume $a,b,G^{\lambda },g\in \tilde{C}^{4}(\mathbf{R})$. Although 
$\mu <1$ in this case and the equation is possibly degenerate, a plausible
convergence rate is still $\kappa =1$ $($or $\kappa =\nu /4$ if $g\in \tilde{%
C}^{\nu }(\mathbf{R}),\nu \in (0,4]$), because the integral part of the
generator of (\ref{fo7}), 
\begin{equation*}
Iv(x)=\int [v(x+G(x)y)-v(x)]\frac{dy}{|y|^{1+\lambda }}=G(x)^{\lambda }\int
[v(x+y)-v(x)]\frac{dy}{|y|^{1+\lambda }},
\end{equation*}%
is differentiable without assuming much about the tail moments of the L\'{e}%
vy measure.



\begin{thebibliography}{99}
\bibitem{cont} Cont, R. and Tankov, P., Financial Modelling with Jump
Processes, Chapman and Hall/CRC Press, 2004.

\bibitem{re} Estrada, R., Vector moment problem for rapidly decreasing
smooth functions of several variables, Proceedings of the American
Mathematical Society, \textbf{126 }(1998), 761-768.

\bibitem{jkmp} Jacod, J., Kurtz, T., M\'{e}l\'{e}ard, S. and Protter, P.,
The approximate Euler scheme for L\'{e}vy driven stochastic differential
equations, Ann. H. Poincare, 41 (2005), 523-558.

\bibitem{jh} Jourdain, B. and Kohatsu-Higa, A., A review of recent results
on approximation of solutions of stochastic differential equations, Progress
in Probability, 65 (2011), 121-144.

\bibitem{KlP00} Kloeden, P. E. and Platen, E., \textit{Numerical Solution of
Stochastic Differential Equations}, Springer Verlag, 2000.

\bibitem{kot} Kohatsu-Higa, A. and Tankov, P., Jump-adapted discretization
schemes for L\'{e}vy-driven SDEs, \textit{Stochastic Processes and their
Applications,} 120 (2010) 2258--2285.

\bibitem{kry} Krylov, N.V., \textit{Controlled Diffusion Processes}\emph{, }%
Springer Verlag, 1980.

\bibitem{KuP01} Kubilius, K. and Platen, E., Rate of Weak Convergence of the
Euler Approximation for Diffusion Processes with Jumps, Quantitative Finance
Research Centre, University of Technology. Sydney, \textit{Research Paper
Series} 54, 2001.

\bibitem{MiP88} Mikulevi\v{c}ius, R. and Platen, E., Time Discrete Taylor
Approximations for It\^{o} Processes with Jump Component, \textit{%
Mathematische Nachrichten} 138 (1988) 93-104.

\bibitem{MiP911} Mikulevi\v{c}ius, R. and Platen, E., Rate of Convergence of
the Euler Approximation for Diffusion Processes, \textit{Mathematische
Nachrichten} 151 (1991) 233-239.

\bibitem{mikprag} Mikulevi\v{c}ius, R. and Pragarauskas, H., On $L_{p}$%
-estimates of some singular integrals related to jump processes,
arXiv:1008.3044v3 [math.PR], 2010.

\bibitem{mikz3} Mikulevi\v{c}ius, R. and Zhang, C., On the rate of
convergence of weak Euler approximation for nondegenerate SDEs driven by L%
\'{e}vy processes, \textit{Stochastic Processes and their Applications,} 121
(2011) 1720-1748.

\bibitem{Mil79} Milstein, G. N., A Method of Second-Order Accuracy
Integration of Stochastic Differential Equations, \textit{Theory of
Probability and its Applications} 23 (1979) 396-401.

\bibitem{Mil86} Milstein, G. N., Weak Approximation of Solutions of Systems
of Stochastic Differential Equations, \textit{Theory of Probability and its
Applications} 30 (1986) 750-766.

\bibitem{Pla99} Platen, E., An Introduction to Numerical Methods for
Stochastic Differential Equations, \textit{Acta Numerica} 8 (1999) 197-246.

\bibitem{PlB10} Platen, E. and Bruti-Liberati, N., \textit{Numerical
Solutions of Stochastic Differential Equations with Jumps in Finance},
Springer Verlag, 2010.

\bibitem{PrT97} Protter, P. E. and Talay, D., The Euler Scheme for L\'{e}vy
Driven Stochastic Differential Equations, \textit{The Annals of Probability}
25 (1997) 393-423.

\bibitem{Tal84} Talay, D., Efficient Numerical Schemes for the Approximation
of Expectations of Functionals of the Solution of a S.D.E. and Applications,
In:\ \textit{Filtering and Control of Random Processes}, \textit{Lecture
Notes in Control and Information Sciences} 61 (1984) 294-313.

\bibitem{Tal86} Talay, D., Discretization of a Stochastic Differential
Equation and Rough Estimate of the Expectations of Functionals of the
Solution, \textit{ESAIM: Mathematical Modelling and Numerical Analysis - Mod%
\'{e}lisation Math\'{e}matique et Analyse Num\'{e}rique} 20 (1986) 141-179.
\end{thebibliography}
\end{document}